\documentclass[reqno]{amsart}
\usepackage{amsmath}
\usepackage{amssymb}
\usepackage{amsfonts}
\usepackage{amsthm}
\usepackage{mathtools}
\usepackage{leftindex}
\usepackage{stmaryrd}
\usepackage{amsthm}
\usepackage{tikz} 
\usetikzlibrary{tqft,calc}
\usetikzlibrary{cd}
\usepackage{enumitem}
\usepackage{graphicx}
\usepackage[left=2.5cm,right=2.5cm,top=3cm,bottom=3cm]{geometry}
\usepackage{tikz-cd}
\usepackage[hidelinks=true]{hyperref}
\usepackage{ulem}

\newcommand{\g}{\mathfrak{g}}

\newcommand{\h}{\mathfrak{h}}

\newcommand{\greg}{\mathfrak{g}_{\mathrm{reg}}}

\makeatletter
\@namedef{subjclassname@1991}{2020 Mathematics Subject Classification}
\makeatother

\numberwithin{equation}{section}

\newtheorem{theorem}{Theorem}[section]
\newtheorem{proposition}[theorem]{Proposition}

\newtheorem{lemma}[theorem]{Lemma}

\newtheorem{mtheorem}{Main Theorem}

\theoremstyle{definition}
\newtheorem{definition}[theorem]{Definition}
\newtheorem{example}[theorem]{Example}
\newtheorem{remark}[theorem]{Remark}

\newcommand {\G}{\mathcal G}

\newcommand {\wh}[1]{\widehat{#1}}

\renewcommand{\H}{\mathcal{H}}
\newcommand{\too}{\longrightarrow}
\newcommand{\Spec}{\operatorname{Spec}}
\newcommand{\s}{\subseteq}
\newcommand{\mto}{\mapsto}
\newcommand{\mtoo}{\longmapsto}
\newcommand{\sll}[1]{\mkern-4mu\mathbin{/\mkern-5mu/}_{\mkern-4mu{#1}}}

\newcommand{\defn}[1]{\textbf{\textit{#1}}}

\makeatletter
\@namedef{subjclassname@2020}{\textup{2020} Mathematics Subject Classification}
\makeatother

\newcommand{\sss}{\mathsf{s}}
\newcommand{\ttt}{\mathsf{t}}
\newcommand{\mmm}{\mathsf{m}}

\newcommand{\uuu}{\mathsf{1}}

\newcommand{\pr}{\mathrm{pr}}
\newcommand{\I}{\mathcal{I}}

\newcommand{\A}{\mathcal{A}}
\newcommand{\B}{\mathcal{B}}
\newcommand{\Cc}{\mathcal{C}}
\newcommand{\fp}[2]{\leftindex_{#1}\times_{#2}\,}

\newcommand{\tto}{\;\substack{\longrightarrow\\[-9pt] \longrightarrow}\;}

\begin{document}
	
	\title{Scheme-theoretic coisotropic reduction}
	
	\author[Peter Crooks]{Peter Crooks}
	\author[Maxence Mayrand]{Maxence Mayrand}
	\address[Peter Crooks]{Department of Mathematics and Statistics\\ Utah State University \\ 3900 Old Main Hill \\ Logan, UT 84322, USA}
	\email{peter.crooks@usu.edu}
	\address[Maxence Mayrand]{D\'{e}partement de math\'{e}matiques \\ Universit\'{e} de Sherbrooke \\ 2500 Bd de l’Universit\'{e} \\ Sherbrooke, QC, J1K 2R1, Canada}
	\email{maxence.mayrand@usherbrooke.ca}
	
	\subjclass{17B63 (primary); 14L30, 53D17 (secondary)} 
	\keywords{coisotropic subvariety, Hamiltonian reduction, Poisson scheme, symplectic groupoid}
	
	\begin{abstract}
		We develop an affine scheme-theoretic version of Hamiltonian reduction by symplectic groupoids. It works over $\Bbbk=\mathbb{R}$ or $\Bbbk=\mathbb{C}$, and is formulated for an affine symplectic groupoid $\G\tto X$, an affine Hamiltonian $\G$-scheme $\mu:M\longrightarrow X$, a coisotropic subvariety $S\s X$, and a stabilizer subgroupoid $\H\tto S$. Our first main result is that the Poisson bracket on $\Bbbk[M]$ induces a Poisson bracket on the subquotient $\Bbbk[\mu^{-1}(S)]^{\H}$. The Poisson scheme  $M\sll{S,\H}\G\coloneqq\mathrm{Spec}(\Bbbk[\mu^{-1}(S)]^{\H})$ is then declared to be a Hamiltonian reduction of $M$. Other main results include sufficient conditions for $M\sll{S,\H}\G$ to inherit a residual Hamiltonian scheme structure.  
		
		Our main results are best viewed as affine scheme-theoretic counterparts to \cite{cro-may:22}, where we simultaneously generalize several Hamiltonian reduction processes. In this way, the present work yields scheme-theoretic analogues of Marsden--Ratiu reduction \cite{mar-rat:86}, Mikami--Weinstein reduction \cite{mik-wei:88}, \'{S}niatycki--Weinstein reduction \cite{sni-wei:83}, and symplectic reduction along general coisotropic submanifolds \cite{cro-may:22}. The initial impetus for this work was its utility in formulating and proving generalizations of the Moore--Tachikawa conjecture.   
	\end{abstract}
	
	\maketitle
	\setcounter{tocdepth}{2}
	\tableofcontents
	\vspace{-10pt}

	
	\section{Introduction}
	
	\subsection{Context}
	Hamiltonian reduction refers to any of several closely related quotient constructions in symplectic and Poisson geometry. Prominent examples include Marsden--Weinstein reduction \cite{mar-wei:74}, Marsden--Ratiu reduction \cite{mar-rat:86}, Mikami--Weinstein reduction \cite{mik-wei:88}, pre-images of Poisson transversals under moment maps \cite{bie:97,cro-roe:22,fre-mar:17}, symplectic cutting \cite{ler:95}, and symplectic implosion \cite{gui-jef-sja:02}. Each of these is a special case of \textit{symplectic reduction along a submanifold} \cite{cro-may:22}, a technique we developed to reduce by a symplectic groupoid along a pre-Poisson submanifold. This technique has versions in the smooth, holomorphic, classical algebro-geometric, and shifted symplectic settings, and contextualizes Ginzburg and Kazhdan's approach \cite{gin-kaz:24} to the Moore--Tachikawa conjecture \cite{moo-tac:11}. Making progress adjacent to this conjecture is one impetus for the present work; see Subsection \ref{Subsection: Relation to the Moore--Tachikawa conjecture} for further details.  
	
	The central purpose of this manuscript is to provide an affine scheme-theoretic counterpart to symplectic reduction along a submanifold. This is accomplished via the two main results described below.
	
	\subsection{First main result} Suppose that $\Bbbk=\mathbb{R}$ or $\mathbb{C}$. Let $\G\tto X$ be an affine symplectic groupoid over $\Bbbk$; this is a groupoid object in the category of affine varieties over $\Bbbk$, together with a multiplicative symplectic structure and other salient properties. These definitions imply that $X$ is an affine Poisson variety. Consider a stabilizer subgroupoid $\H\tto S$ over a smooth, closed, coisotropic subvariety $S\s X$; see Subsection \ref{Subsection: Stabilizer subgroupoids} for a precise definition of \textit{stabilizer subgroupoid}.\footnote{They are examples of 1-shifted Lagrangians on $\G$ in the sense of \cite{ptvv:13}; see \cite[Section 6]{cro-may:22}.} Given an affine Hamiltonian $\G$-scheme $\mu:M\longrightarrow X$ over $\Bbbk$, the closed subscheme $\mu^{-1}(S)\s M$ carries an action of $\H$. Write $j:\mu^{-1}(S)\longrightarrow M$ for the inclusion and $\Bbbk[\mu^{-1}(S)]^{\H}\s \Bbbk[\mu^{-1}(S)]$ for the subalgebra of elements fixed by $\H$. Our first main result is the following.
	
	\begin{mtheorem}\label{Theorem: Main}
		There is a unique Poisson bracket on $\Bbbk[\mu^{-1}(S)]^{\H}$ satisfying $\{j^*f_1,j^*f_2\}=j^*\{f_1,f_2\}$ for all $f_1,f_2\in\Bbbk[M]$ with $j^*f_1,j^*f_2\in\Bbbk[\mu^{-1}(S)]^{\H}$.
	\end{mtheorem}
	
	We call the affine Poisson scheme 
	$$M\sll{S,\H}\G\coloneqq\mathrm{Spec}(\Bbbk[\mu^{-1}(S)]^{\H})$$
	the \textit{Hamiltonian reduction of $M$ by $\G$ along $S$ with respect to $\H$}.
	
	\subsection{Second main result}
	Our Hamiltonian reductions inherit residual symplectic groupoid actions under reasonable hypotheses. In more detail, let $\G\tto X$ and $\I\tto Y$ be affine symplectic groupoids over $\Bbbk$. Suppose that $\H\tto S$ is a stabilizer subgroupoid of $\I$ over a smooth, closed, coisotropic subvariety $S\s Y$.  Let us also consider an affine Poisson scheme $M$ over $\Bbbk$. Assume that $M$ is simultaneously an affine $\G$-scheme $\mu:M\longrightarrow X$ and an affine $\I$-scheme $\nu:M\longrightarrow Y$, in such a way that the two actions commute. Write $\mathcal{A}^*:\Bbbk[M]\longrightarrow\Bbbk[\G]\otimes_{\Bbbk[X]}\Bbbk[M]$ for the algebra morphism induced by the $\G$-action morphism $\mathcal{A}:\G\fp{\sss}{\mu} M\longrightarrow M$, and $\overline{\mu}:M\sll{S,\H}\I\longrightarrow X$ for descent of the restriction of $\mu$ to $\nu^{-1}(S)\s M$.
	
	\begin{mtheorem}\label{Theorem: Main2}
		The Hamiltonian $\G$-action on $M$ descends to a Hamiltonian $\G$-action on $M \sll{S,\H} \I$. More precisely, the following statements are true.
		\begin{itemize}
			\item[\textup{(i)}] The morphism $\mathcal{A}^*$ descends to a morphism $\overline{\mathcal{A}^*}:\Bbbk[\nu^{-1}(S)]\longrightarrow\Bbbk[\G]\otimes_{\Bbbk[X]}\Bbbk[M]$.
			\item[\textup{(ii)}] One has $\overline{\mathcal{A}^*}(\Bbbk[\nu^{-1}(S)]^{\H})\s\Bbbk[\G]\otimes_{\Bbbk[X]}\Bbbk[\nu^{-1}(S)]^{\H}$.
			\item[\textup{(iii)}] Restrict $\overline{\mathcal{A}^*}$ to an algebra morphism $\Bbbk[\nu^{-1}(S)]^{\H}\longrightarrow\Bbbk[\G]\otimes_{\Bbbk[X]}\Bbbk[\nu^{-1}(S)]^{\H}$. This restriction corresponds to an affine $\G$-scheme structure $\G\fp{\sss}{\overline{\mu}}(M\sll{S,\H}\I)\longrightarrow M\sll{S,\H}\I$ on $\overline{\mu}:M\sll{S,\H}\I\longrightarrow X$.
			\item[\textup{(iv)}] The affine $\G$-scheme structure in \textup{(iii)} is Hamiltonian with respect to the Poisson structure on $M\sll{S,\H}\I$ from Main Theorem \ref{Theorem: Main}.
		\end{itemize}
	\end{mtheorem}
	
	Parts (i)--(iii) hold for any affine algebraic Lie groupoids $\G$ and $\I$, any closed subgroupoid $\H$ of $\I$, and any affine scheme $M$ with commuting actions of $\G$ and $\I$. The Poisson-geometric hypotheses are only relevant to Part (iv).
	
	\subsection{Comparisons with other Hamiltonian reduction techniques}
	Main Theorem \ref{Theorem: Main} yields scheme-theoretic counterparts to several Hamiltonian reduction procedures in the literature. The following are some examples.
	
	\subsubsection*{Marsden--Ratiu reduction}
	Let $G$ be a Lie group with Lie algebra $\g$. Suppose that a Poisson manifold $M$ is equipped with a Hamiltonian $G$-action and moment map $\mu:M\longrightarrow\g^*$. Assume that $G$ acts freely and properly on $\mu^{-1}(0)$. It follows that $\mu^{-1}(0)$ is a submanifold of $M$, and that $\mu^{-1}(0)/G$ carries a unique smooth manifold structure for which the quotient map $\pi:\mu^{-1}(0)\longrightarrow\mu^{-1}(0)/G$ is a submersion. The pullback $\pi^*:C^{\infty}(\mu^{-1}(0)/G)\longrightarrow C^{\infty}(\mu^{-1}(0))$ is injective with image $ C^{\infty}(\mu^{-1}(0))^G$; it thereby gives an identification $C^{\infty}(\mu^{-1}(0)/G)=C^{\infty}(\mu^{-1}(0))^G$. The Marsden--Ratiu Poisson structure on $\mu^{-1}(0)/G$ then amounts to $C^{\infty}(\mu^{-1}(0))^G$ carrying a unique Poisson bracket that satisfies the following condition: $\{j^*f_1,j^*f_2\}=j^*\{f_1,f_2\}$ for all $f_1,f_2\in C^{\infty}(M)$ with $j^*f_1,j^*f_2\in C^{\infty}(\mu^{-1}(0))^G$, where $j:\mu^{-1}(0)\longrightarrow M$ is the inclusion map \cite{mar-rat:86}. This extends Marsden--Weinstein reduction \cite{mar-wei:74} to the setting of Poisson manifolds.
	
	Now let $G$ be an affine algebraic group over $\Bbbk=\mathbb{R}$ or $\mathbb{C}$. Suppose that $M$ is an affine Hamiltonian $G$-scheme with moment map $\mu:M\longrightarrow\g^*$. In analogy with Marsden--Ratiu reduction, $\Bbbk[\mu^{-1}(0)]^G$ inherits a Poisson bracket from $\Bbbk[M]$. This scheme-theoretic counterpart of Marsden--Ratiu reduction is a special case of Main Theorem \ref{Theorem: Main}; one simply notes that $M$ is an affine Hamiltonian $T^*G$-scheme for the cotangent groupoid $\G=T^*G$, takes $S=\{0\}\s\g^*$, and lets $\H$ be the isotropy group of $0$ in $\G$.
	
	\subsubsection*{\'{S}niatycki--Weinstein reduction}
	Let $G$ be a connected Lie group with Lie algebra $\g$. Suppose that $M$ is a symplectic manifold with a Hamiltonian action of $G$ and moment map $\mu:M\longrightarrow\g^*$. Write $\mu^{\xi}\in C^{\infty}(M)$ for the pointwise pairing of $\mu$ with any $\xi\in\g$, and $I\coloneqq\langle\mu^{\xi}:\xi\in\g\rangle\s C^{\infty}(M)$ for the ideal generated by all such pairings. Since $I$ is invariant under the action of $G$ on $C^{\infty}(M)$, the quotient algebra $C^{\infty}(M)/I$ inherits a $G$-action. \'{S}niatycki and Weinstein show $(C^{\infty}(M)/I)^G$ to have a unique Poisson bracket with the following property: $\{\pi(f_1),\pi(f_2)\}=\pi(\{f_1,f_2\})$ for all $f_1,f_2\in C^{\infty}(M)$ satisfying $\pi(f_1),\pi(f_2)\in (C^{\infty}(M)/I)^G$, where $\pi:C^{\infty}(M)\longrightarrow C^{\infty}(M)/I$ is the quotient map. If $G$ acts freely and properly on $\mu^{-1}(0)$, then $(C^{\infty}(M)/I)^G$ and $C^{\infty}(\mu^{-1}(0)/G)$ are explicitly isomorphic as Poisson algebras.
	
	The \'{S}niatycki--Weinstein setup is very natural from a scheme-theoretic standpoint. This is witnessed by letting $G$ be affine algebraic over $\Bbbk=\mathbb{R}$ or $\mathbb{C}$, and $M$ be an affine Hamiltonian $G$-scheme with moment map $\mu:M\longrightarrow\g^*$. Since $I\coloneqq\langle\mu^{\xi}:\xi\in\g\rangle\s \Bbbk[M]$ is the ideal of the closed subscheme $\mu^{-1}(0)\s M$, one has $\Bbbk[\mu^{-1}(0)]=\Bbbk[M]/I$. One also finds that $\Bbbk[\mu^{-1}(0)]^G=(\Bbbk[M]/I)^G$ inherits a Poisson bracket from $\Bbbk[M]$, yielding a scheme-theoretic counterpart of \'{S}niatycki--Weinstein reduction. This counterpart can be deduced by applying Main Theorem \ref{Theorem: Main} with $\G=T^*G$, $S=\{0\}\s\g^*$, and $\H$ the isotropy group of $0$ in $\G$.
	
	\subsubsection*{Symplectic reduction along a pre-Poisson subvariety}     
	Consider the following rough summary of \cite[Theorem 5.3]{cro-may:22}. Let $\G\tto X$ be an algebraic symplectic groupoid over $\mathbb{C}$. Suppose that $\H\tto S$ is a stabilizer subgroupoid over a pre-Poisson subvariety $S\s X$, and that $\mu:M\longrightarrow X$ is a symplectic Hamiltonian $\G$-variety. Assume that $\mu^{-1}(S)$ is reduced as a subscheme of $M$. Let us also consider an algebraic variety $Q$ and $\H$-invariant variety morphism $\pi:\mu^{-1}(S)\longrightarrow Q$ for which $\mathcal{O}_Q\longrightarrow(\pi_*\mathcal{O}_{\mu^{-1}(S)})^{\mathcal{H}}$ is an isomorphism. The conclusion of \cite[Theorem 5.3]{cro-may:22} is that $\mathcal{O}_M$ induces a Poisson bracket on $\mathcal{O}_{Q}$; see the theorem itself for a more precise statement. The Poisson variety $Q$ is called the \textit{Hamiltonian reduction of $M$ by $\G$ along $S$ with respect to $\H$ and $\pi$}.
	
	Now consider \cite[Theorem 5.3]{cro-may:22} in the following special case: $\G$, $M$, and $X$ are affine, $S$ is closed and coisotropic in $X$, and $\mathbb{C}[\mu^{-1}(S)]^{\H}$ is finitely generated. One may take $Q$ to be $\mathrm{Spec}(\mathbb{C}[\mu^{-1}(S)]^{\H})$ and $\pi$ to be induced by the inclusion $\mathbb{C}[\mu^{-1}(S)]^{\H}\longrightarrow\mathbb{C}[\mu^{-1}(S)]$. The Poisson variety $Q=\mathrm{Spec}(\mathbb{C}[\mu^{-1}(S)]^{\H})$ in \cite[Theorem 5.3]{cro-may:22} is $M\sll{S,\H}\G$, as defined immediately following Main Theorem \ref{Theorem: Main}. In other words, Main Theorem \ref{Theorem: Main} is a scheme-theoretic generalization of the special case of \cite[Theorem 5.3]{cro-may:22} being considered.
	
	\subsection{Relation to the Moore--Tachikawa conjecture}\label{Subsection: Relation to the Moore--Tachikawa conjecture} The Moore--Tachikawa conjecture \cite{moo-tac:11} is central to ongoing research at the interface of geometric representation theory, Poisson geometry, and theoretical physics \cite{ara:18,bie:23,bra-fin-nak:19,cal:15,dan-kir-mar:24,gin-kaz:24}. It posits the existence of certain two-dimensional topological quantum field theories (TQFTs) valued in $\mathrm{HS}$, the symmetric monoidal category of \textit{holomorphic symplectic varieties with Hamiltonian actions}. Complex reductive  groups constitute the objects of $\mathrm{HS}$, while morphisms are isomorphism classes of certain complex affine symplectic varieties with Hamiltonian actions; see \cite[Section 3.1]{moo-tac:11} for more details. 
	
	To formulate the Moore--Tachikawa conjecture, let $\mathrm{COB}_2$ denote the symmetric monoidal category of two-dimensional cobordisms. A two-dimensional TQFT valued in $\mathrm{HS}$ is a symmetric monoidal functor $\mathrm{COB}_2\longrightarrow\mathrm{HS}$. On the other hand, fix a connected simple affine algebraic group $G$ over $\mathbb{C}$. Let $\mathcal{S}\s\g$ be the Slodowy slice associated to a principal $\mathfrak{sl}_2$-triple in $\g$, the Lie algebra of $G$. As an affine symplectic Hamiltonian $G$-variety, $G\times\mathcal{S}$ is a morphism in $\mathrm{HS}$ from $G$ to the trivial group. Moore and Tachikawa's conjecture is that there exists a two-dimensional TQFT $\eta:\mathrm{COB}_2\longrightarrow\mathrm{HS}$ satisfying $\eta(S^1)=G$ and $\eta(C)=G\times\mathcal{S}$, where $C$ is the standard two-dimensional cobordism from $S^1$ to $\emptyset$. Ginzburg and Kazhdan reduce this conjecture to proving that certain algebras are finitely generated \cite{gin-kaz:24}. Such algebras are known to be finitely generated in Lie type $A$, as implied by the work Braverman--Finkelberg--Nakajima \cite{bra-fin-nak:19}.
	
	It is reasonable to hope that a broadened Moore--Tachikawa conjecture would motivate the discovery of new TQFTs. One might also hope that this process would uncover subtleties and insights relevant to proving the original conjecture. These premises underlie our manuscripts \cite{cro-may:24} and \cite{cro-may:242}. We enlarge $\mathrm{HS}$ to a category with complex affine symplectic groupoids as objects, and morphisms consisting of isomorphism classes of certain complex affine Poisson schemes with Hamiltonian symplectic groupoid actions. This allows us to formulate a generalized Moore--Tachikawa conjecture and prove several cases thereof. Main Theorems \ref{Theorem: Main} and \ref{Theorem: Main2} play crucial roles in defining the enlargement of $\mathrm{HS}$. 
	
	\subsection{Organization} Each section is preceded by a description of its contents. Section \ref{Section: Properties} establishes the salient definitions and results on scheme-theoretic groupoid actions. The proofs of Main Theorems \ref{Theorem: Main} and \ref{Theorem: Main2} are given in Sections \ref{Section: Coisotropic reduction} and \ref{Section: Residual actions}, respectively.  
	
	\subsection*{Acknowledgements}
	We thank the referee for providing feedback that led to an improved manuscript. P.C. was partially supported by the National Science Foundation Grant DMS-2454103 and Simons Foundation Grant MPS-TSM-00002292. M.M. was partially supported by the NSERC Discovery Grant RGPIN-2023-04587. 
	
	\section{Affine algebraic Lie groupoids}\label{Section: Properties}
	This section develops the salient ideas on affine algebraic Lie groupoids and their actions on affine schemes. It is principally intended to serve as preparation for the Poisson-geometric content of subsequent sections. As such, it gives purely commutative-algebraic formulations of groupoid actions, induced actions of subgroupoids on subschemes, residual groupoid actions on quotients, graphs of groupoid actions, and product groupoid actions.
	
	\subsection{Fundamental conventions}
	Take $\Bbbk$ to be $\mathbb{R}$ or $\mathbb{C}$. This fixed field $\Bbbk$ is understood to be the base field underlying all notions for which a base field is presupposed, e.g. vector spaces, algebras, algebraic varieties, etc. All schemes are understood to be over $\Bbbk$. We write $\Bbbk[X]$ for the global sections of the structure sheaf of a scheme $X$. Given a scheme morphism $f:X\longrightarrow Y$, we write $f^*:\Bbbk[Y]\longrightarrow\Bbbk[X]$ for the induced morphism on global sections. The term \textit{algebraic variety} is used for a reduced scheme of finite type; it need not be irreducible.
	
	\subsection{Algebraic groupoids} We adhere to the following conventions regarding groupoids. An \textit{algebraic groupoid} is defined to be a groupoid object $\G\tto X$ in the category of algebraic varieties. Throughout this manuscript, $\sss:\G\longrightarrow X$, $\ttt:\G\longrightarrow X$, $\mmm:\G\fp{\sss}{\ttt}\G\longrightarrow\G$, and $\uuu:X\longrightarrow\G$ are used to denote the underlying source, target, multiplication, and identity bisection morphisms, respectively. The term \textit{algebraic Lie groupoid} is used if $\G$ and $X$ are smooth varieties, and $\sss$ and $\ttt$ are smooth morphisms. An \textit{affine algebraic Lie groupoid} is an algebraic Lie groupoid $\G\tto X$ for which $\G$ and $X$ are affine varieties.
	
	\subsection{Actions} This manuscript makes extensive use of groupoid actions in a scheme-theoretic context. To establish the relevant foundations, recall the equivalence of categories between affine schemes and commutative algebras. One can thereby rephrase the definition of a (left) affine group scheme action on an affine scheme in terms of coordinate algebras. The usual definition of a groupoid action can be translated analogously; it yields the following definition of an \textit{affine $\G$-scheme}, where $\G \tto X$ is an affine algebraic Lie groupoid.   
	
	\begin{definition} \label{s469vs1d}
		An \textit{affine $\G$-scheme} is a triple $(M,\mu,\mathcal{A})$ of an affine scheme $M$, morphism $\mu:M\longrightarrow X$, and morphism $\mathcal{A}:\G\fp{\sss}{\mu}M\longrightarrow M$ that satisfy the following properties.
		\begin{enumerate}[label={\textup{(\roman*)}}]
			\item \label{age2txk9} The diagram
			$$\begin{tikzcd}
				\Bbbk[X]\arrow[r,"\mu^*"] \arrow[d,"\ttt^*"] & \Bbbk[M]\arrow[d,"\mathcal{A}^*"] \\
				\Bbbk[\G]\arrow[r, "\mathrm{id}\otimes 1"] & \Bbbk[\G]\otimes_{\Bbbk[X]}\Bbbk[M]\\
			\end{tikzcd}$$
			commutes.
			\item The composite morphism
			$$\Bbbk[M]\overset{\mathcal{A}^*}\longrightarrow\Bbbk[\G]\otimes_{\Bbbk[X]}\Bbbk[M]\overset{\uuu^*\otimes\mathrm{id}}\longrightarrow\Bbbk[X]\otimes\Bbbk[M]\overset{\mu^*\mathrm{id}}\longrightarrow\Bbbk[M]$$ is the identity on $\Bbbk[M]$, where $(\mu^*\mathrm{id})(f\otimes h)=\mu^*(f)h$ for all $f\in\Bbbk[M]$ and $h\in\Bbbk[X]$.
			\item The diagram
			$$\begin{tikzcd}
				\Bbbk[M]\arrow[r,"\mathcal{A}^*"] \arrow[d,"\mathcal{A}^*"] & \Bbbk[\G]\otimes_{\Bbbk[X]}\Bbbk[M]\arrow[d,"\mmm^*\otimes\mathrm{id}"] \\
				\Bbbk[\G]\otimes_{\Bbbk[X]}\Bbbk[M]\arrow[r, "\mathrm{id}\otimes\mathcal{A}^*"] & \Bbbk[\G]\otimes_{\Bbbk[X]}\otimes\Bbbk[\G]\otimes_{\Bbbk[X]}\Bbbk[M]\\
			\end{tikzcd}$$
			commutes.
		\end{enumerate}
		In this case, we call $\mathcal{A}$ and $\mu$ the \textit{action morphism} and \textit{moment map}, respectively.
	\end{definition}
	
	\subsection{Equivariance and invariance}
	Suppose that an affine algebraic group $G$ acts algebraically on affine varieties $M$ and $N$. Let $\mathcal{A}:G\times M\longrightarrow M$ and $\mathcal{B}:G\times N\longrightarrow N$ denote the corresponding action morphisms. Note that a morphism $\psi:M\longrightarrow N$ is $G$-equivariant if and only if $\psi\circ\mathcal{A}=\mathcal{B}\circ(\mathrm{id},\psi)$, where $\mathrm{id}:G\longrightarrow G$ is the identity map. This is equivalent to the condition $\mathcal{A}^*\circ\psi^*=(\mathrm{id}\otimes\psi^*)\circ\mathcal{B}^*$. Similarly, a morphism $f:M\longrightarrow\Bbbk$ is $G$-invariant if and only if $\mathcal{A}^*f=1\otimes f$.
	
	Now let $\G\tto X$ be an affine algebraic Lie groupoid. The preceding discussion generalizes as follows.
	
	\begin{definition}
		Let $(M,\mu,\mathcal{A})$ and $(N,\nu,\mathcal{B})$ be affine $\G$-schemes. A morphism $\psi:M\longrightarrow N$ is called \textit{$\G$-equivariant} if it satisfies the following properties. 
		\begin{itemize} 
			\item[\textup{(i)}] The diagram $$\begin{tikzcd}
				\Bbbk[N]\arrow[r,"\psi^*"] \arrow[d,"\mathcal{B}^*"] & \Bbbk[M]\arrow[d,"\mathcal{A}^*"] \\
				\Bbbk[\G]\otimes_{\Bbbk[X]}\Bbbk[N]\arrow[r, "\mathrm{id}\otimes\psi^*"] & \Bbbk[\G]\otimes_{\Bbbk[X]}\Bbbk[M]\\
			\end{tikzcd}$$
			commutes.
			\item[\textup{(ii)}] The diagram $$\begin{tikzcd}
				M\arrow[r,"\psi"]\arrow[swap,dr,"\mu"]  & N\arrow[d,"\nu"] \\
				& X\\
			\end{tikzcd}$$
			commutes.
		\end{itemize}
	\end{definition}
	
	\begin{definition}
		Let $(M,\mu,\mathcal{A})$ be an affine $\G$-scheme. The \textit{algebra of $\G$-invariants} is the subalgebra $\Bbbk[M]^{\G}\coloneqq\{f\in\Bbbk[M]:\mathcal{A}^*f=1\otimes f\}\s\Bbbk[M]$. 
	\end{definition}
	
	These two definitions feature in the following result.
	
	\begin{lemma}\label{Lemma: Invariants}
		Let $(M,\mu,\mathcal{A})$ and $(N,\nu,\mathcal{B})$ be affine $\G$-schemes.
		\begin{itemize}
			\item[\textup{(i)}] If $\psi:M\longrightarrow N$ is a $\G$-equivariant morphism, then $\psi^*(\Bbbk[N]^{\G})\s\Bbbk[M]^{\G}$.
			\item[\textup{(ii)}] Suppose that $\G$ acts trivially on $N$. Equip $M\times N$ with the affine $\G$-scheme structure coming from the diagonal $\G$-action on $M\times N$. We then have $\Bbbk[M\times N]^{\G}=\Bbbk[M]^{\G}\otimes\Bbbk[N]$.
		\end{itemize}
	\end{lemma}
	
	\begin{proof}
		We first prove (i). To this end, suppose that $f\in\Bbbk[N]^{\G}$. We have
		$$\mathcal{A}^*\psi^*f=(\mathrm{id}\otimes\psi^*)(\mathcal{B}^*f)=(\mathrm{id}\otimes\psi^*)(1\otimes f)=1\otimes\psi^*f.$$ This completes the proof of (i).
		
		We now verify (ii). Note that the $\G$-action on $M$ and trivial $\G$-action on $N$ determine an affine $\G\times\G$-scheme structure on $M\times N$. Since $\G$ acts trivially on $N$, $\Bbbk[M\times N]^{\G\times\G}$ coincides with the algebra $\Bbbk[M\times N]^{\G}$ in the statement of the lemma. It therefore suffices to prove that $\Bbbk[M\times N]^{\G\times\G}=\Bbbk[M]^{\G}\otimes\Bbbk[N]$. 
		
		Suppose that $f\in\Bbbk[M\times N]=\Bbbk[M]\otimes\Bbbk[N]$. Choose $n\in\mathbb{Z}_{\geq 0}$, elements $\alpha_1,\ldots,\alpha_n\in\Bbbk[M]$, and linearly independent elements $\beta_1,\ldots,\beta_n\in\Bbbk[N]$ satisfying $f=\sum_{i=1}^n(\alpha_i\otimes\beta_i)$. Let us also consider the action morphism $\mathcal{C}:(\G\times\G)\fp{(\sss,\sss)}{(\mu,\nu)}(M\times N)\longrightarrow M\times N$ and canonical isomorphism $\theta:(\G\times\G)\fp{(\sss,\sss)}{(\mu,\nu)}(M\times N)\longrightarrow(\G\fp{\sss}{\mu}M)\times (\G\fp{\sss}{\nu}N)$. We have the commutative diagram 
		$$\begin{tikzcd}
			(\G\times\G)\fp{(\sss,\sss)}{(\mu,\nu)}(M\times N) \arrow[rr, "\mathcal{C}"]\arrow[dr, "\theta"] && M\times N\\
			& (\G\fp{\sss}{\mu}M)\times (\G\fp{\sss}{\nu}N)\arrow[ur, "{(\mathcal{A},\mathcal{B})}"] & 
		\end{tikzcd}.$$ One has the induced commutative diagram
		$$\begin{tikzcd}
			& (\Bbbk[\G]\otimes_{\Bbbk[X]}\Bbbk[M])\otimes(\Bbbk[\G]\otimes_{\Bbbk[X]}\Bbbk[N])\arrow[dr, "\theta^*"]\\
			\Bbbk[M]\otimes\Bbbk[N]\arrow[rr,"\mathcal{C}^*"]\arrow[ur,"\mathcal{A}^*\otimes\mathcal{B}^*"] && (\Bbbk[\G]\otimes\Bbbk[\G])\otimes_{\Bbbk[X]\otimes\Bbbk[X]}(\Bbbk[M]\otimes\Bbbk[N])
		\end{tikzcd}.$$ Our hypotheses also imply that $\mathcal{B}^*\beta_i=1\otimes\beta_i$ for all $i=1,\ldots,n$. With these last two sentences in mind, it follows that
		\begin{align*}
			f\in\Bbbk[M\times N]^{\G\times\G} & \Longleftrightarrow \mathcal{C}^*f=1\otimes f\\
			& \Longleftrightarrow \theta^*((\mathcal{A}^*\otimes\mathcal{B}^*)(f))=1\otimes f\\
			& \Longleftrightarrow \theta^*\left(\sum_{i=1}^n(\mathcal{A}^*\alpha_i\otimes\mathcal{B}^*\beta_i)\right)=1\otimes\left(\sum_{i=1}^n(\alpha_i\otimes\beta_i)\right)\\
			& \Longleftrightarrow \theta^*\left(\sum_{i=1}^n(\mathcal{A}^*\alpha_i\otimes(1\otimes\beta_i))\right)=1\otimes\left(\sum_{i=1}^n(\alpha_i\otimes\beta_i)\right)\\
			& \Longleftrightarrow \sum_{i=1}^n(\mathcal{A}^*\alpha_i\otimes(1\otimes\beta_i)) = \sum_{i=1}^n((1\otimes\alpha_i)\otimes(1\otimes\beta_i)).
		\end{align*}
		We may also use the linear independence of $\beta_1,\ldots,\beta_n\in\Bbbk[N]$ and injectivity of $$\Bbbk[N]\longrightarrow\Bbbk[\G]\otimes_{\Bbbk[X]}\Bbbk[N],\quad h\mapsto 1\otimes h$$ to conclude that $1\otimes\beta_1,\ldots,1\otimes\beta_n\in\Bbbk[\G]\otimes_{\Bbbk[X]}\Bbbk[N]$ are linearly independent. These last two sentences imply that $f\in\Bbbk[M\times N]^{\G\times\G}$ if and only if $\mathcal{A}^*\alpha_i=1\otimes\alpha_i$ for all $i=1,\ldots,n$. This completes the proof of (ii).
	\end{proof}
	
	\subsection{Actions of subgroupoids on subschemes}\label{Subsection: Actions of subgroupoids on subschemes}
	
	In this subsection, we give a precise context in which a subgroupoid acts on a subscheme. Consider an affine algebraic Lie groupoid $\G\tto X$ and affine $\G$-scheme $\mu:M\longrightarrow X$. The action morphism $\mathcal{A}:\G\fp{\sss}{\mu} M\longrightarrow M$ determines an algebra morphism $$\mathcal{A}^*:\Bbbk[M]\longrightarrow\Bbbk[\G]\otimes_{\Bbbk[X]}\Bbbk[M].$$ On the other hand, suppose that $\mathcal{H}\tto S$ is a closed subgroupoid of $\G$ over a closed subvariety $S\s X$. Write $i:S\longrightarrow X$ for the inclusion and $\overline{\mu}:\mu^{-1}(S)\longrightarrow S$ for the pullback of $\mu$ along $i$. Denote by $\overline{\sss}:\H\longrightarrow S$, $\overline{\ttt}:\H\longrightarrow S$, $\overline{\mmm}:\H\fp{\overline{\sss}}{\overline{\ttt}}\H\longrightarrow\H$, and
	$\overline{\uuu}:S\longrightarrow\H$ the source, target, multiplication, and identity bisection morphisms for $\H$, respectively. Let us also consider the inclusions $j:\mu^{-1}(S)\longrightarrow M$ and $k:\mathcal{H}\longrightarrow\G$, and algebra morphism
	$$\varphi\coloneqq (k^*\otimes j^*)\circ\mathcal{A}^*:\Bbbk[M]\longrightarrow\Bbbk[\mathcal{H}]\otimes_{\Bbbk[S]}\Bbbk[\mu^{-1}(S)].$$
	
	\begin{proposition}\label{Proposition: Action on subschemes}
		The following statements are true.
		\begin{itemize}
			\item[\textup{(i)}] The morphism $\varphi$ descends under $j^*:\Bbbk[M]\longrightarrow\Bbbk[\mu^{-1}(S)]$ to a morphism $\overline{\varphi}:\Bbbk[\mu^{-1}(S)]\longrightarrow\Bbbk[\mathcal{H}]\otimes_{\Bbbk[S]}\Bbbk[\mu^{-1}(S)]$.
			\item[\textup{(ii)}] One has $\overline{\varphi}=\mathcal{B}^*$ for a unique affine $\mathcal{H}$-scheme structure $\mathcal{B}:\mathcal{H}\fp{\overline{\sss}}{\overline{\mu}}\mu^{-1}(S)\longrightarrow\mu^{-1}(S)$ on $\overline{\mu}:\mu^{-1}(S)\longrightarrow S$.
		\end{itemize}
	\end{proposition}
	
	\begin{proof}
		To prove (i), suppose that $f\in\Bbbk[X]$ satisfies $i^*f=0$. We have the commutative diagrams
		$$\begin{tikzcd}
			\Bbbk[X]\arrow[r,"\ttt^*\otimes 1"] \arrow[d,"\mu^*"] & \Bbbk[\G]\otimes_{\Bbbk[X]}\Bbbk[X]\arrow[d,"\mathrm{id}\otimes\mu^*"] \\
			\Bbbk[M]\arrow[r, "\mathcal{A}^*"]\arrow[dr,"\varphi"] & \Bbbk[\G]\otimes_{\Bbbk[X]}\Bbbk[M]\arrow[d,"k^*\otimes j^*"]\\
			& \Bbbk[\mathcal{H}]\otimes_{\Bbbk[S]}\Bbbk[\mu^{-1}(S)]
		\end{tikzcd}\quad\text{and}\quad\begin{tikzcd}
			\Bbbk[X]\arrow[r,"i^*"] \arrow[d,"\ttt^*"] & \Bbbk[S]\arrow[d,"\overline{\ttt}^*"] \\
			\Bbbk[\G]\arrow[r, "k^*"] & \Bbbk[\mathcal{H}]\\
		\end{tikzcd}.$$ It follows that
		\begin{align*}
			\varphi(\mu^*f) & = (k^*\otimes j^*)(\mathcal{A}^*\mu^*f) \\
			& = (k^*\otimes j^*)((\mathrm{id}\otimes\mu^*)((\ttt^*\otimes 1)(f)))\\
			& = k^*\ttt^*f\otimes 1\\
			& = \overline{\ttt}^*i^*f\otimes 1\\
			& = 0.
		\end{align*}
		This calculation shows the kernel of $j^*$ to be contained in that of $\varphi$. The proof of (i) is therefore complete.
		
		We now verify (ii), i.e. that $\overline{\varphi}$ satisfies the relevant properties in Definition \ref{s469vs1d}. A first step is to consider the diagram
		$$\begin{tikzcd}
			\Bbbk[X]\arrow[r, "\mu^*"]\arrow[d,"i^*"] & \Bbbk[M]\arrow[d,"j^*"]\\
			\Bbbk[S]\arrow[r,"\overline{\mu}^*"]\arrow[d,"\overline{\ttt}^*"] & \Bbbk[\mu^{-1}(S)]\arrow[d,"\overline{\varphi}"]\\
			\Bbbk[\mathcal{H}]\arrow[r,"\mathrm{id}\otimes 1"] & \Bbbk[\mathcal{H}]\otimes_{\Bbbk[S]}\Bbbk[\mu^{-1}(S)] \\
			\Bbbk[\G]\arrow[r,"\mathrm{id}\otimes 1"]\arrow[u, swap,"k^*"] & \Bbbk[\mathcal{G}]\otimes_{\Bbbk[X]}\Bbbk[M]\arrow[u,swap,"k^*\otimes j^*"]
		\end{tikzcd}.$$
		The top and bottom squares commute. To check that the middle square also commutes, suppose that $f\in\Bbbk[S]$. Choose $h\in\Bbbk[X]$ satisfying $f=i^*h$. We have 
		\begin{align*}
			\overline{\varphi}(\overline{\mu}^*f) & = \overline{\varphi}(j^*\mu^*h)\\
			& = \varphi(\mu^*h)\\
			& = (k^*\otimes j^*)(\mathcal{A}^*\mu^*h)\\
			& = (k^*\otimes j^*)((\mathrm{id}\otimes 1)(\ttt^*h))\\
			& = k^*\ttt^*h\otimes 1\\
			& = \overline{\ttt}^*i^*h\otimes 1\\
			& = (\mathrm{id}\otimes 1)(\overline{\ttt}^*f),		
		\end{align*}
		as required.

		Now consider the commuting squares
		$$\begin{tikzcd}
			\Bbbk[M]\arrow[r,"\mathcal{A}^*"]\arrow[d,"j^*"] & \Bbbk[\G]\otimes_{\Bbbk[X]}\Bbbk[M]\arrow[r,"\uuu^*\otimes\mathrm{id}"]\arrow[d,"k^*\otimes j^*"] & \Bbbk[X]\otimes\Bbbk[M]\arrow[r,"\mu^*\mathrm{id}"]\arrow[d,"i^*\otimes j^*"] & \Bbbk[M]\arrow[d,"j^*"]\\
			\Bbbk[\mu^{-1}(S)]\arrow[r,"\overline{\varphi}"] & \Bbbk[\mathcal{H}]\otimes_{\Bbbk[S]}\Bbbk[\mu^{-1}(S)]\arrow[r,"\overline{\uuu}^*\otimes\mathrm{id}"] & \Bbbk[S]\otimes\Bbbk[\mu^{-1}(S)]\arrow[r,"\overline{\mu}^*\mathrm{id}"] & \Bbbk[\mu^{-1}(S)]
		\end{tikzcd}.$$
		Note that composing the horizontal arrows in the top row yields the identity map on $\Bbbk[M]$. It follows that this identity map is a lift under $j^*$ of what results from composing the horizontal arrows in the bottom row. In other words, composing the horizontal arrows in the bottom row yields the identity map on $\Bbbk[\mu^{-1}(S)]$.
		
		Let us now consider the commutative diagram 
		$$\begin{tikzcd}
			\Bbbk[M]\arrow[r,"\mathcal{A}^*"] \arrow[d,"\mathcal{A}^*"] & \Bbbk[\G]\otimes_{\Bbbk[X]}\Bbbk[M]\arrow[d,"\mmm^*\otimes\mathrm{id}"] \\
			\Bbbk[\G]\otimes_{\Bbbk[X]}\Bbbk[M]\arrow[r, "\mathrm{id}\otimes\mathcal{A}^*"] & \Bbbk[\G]\otimes_{\Bbbk[X]}\otimes\Bbbk[\G]\otimes_{\Bbbk[X]}\Bbbk[M]\\
		\end{tikzcd}$$
		and diagram
		$$\begin{tikzcd}
			\Bbbk[M]\arrow[r, "\mathcal{A}^*"]\arrow[d,"j^*"] & \Bbbk[\G]\otimes_{\Bbbk[X]}\Bbbk[M]\arrow[d,"k^*\otimes j^*"]\\
			\Bbbk[\mu^{-1}(S)]\arrow[r,"\overline{\varphi}^*"]\arrow[d,"\overline{\varphi}"] & \Bbbk[\mathcal{H}]\otimes_{\Bbbk[S]}\Bbbk[\mu^{-1}(S)]\arrow[d,"\overline{\mmm}^*\otimes\mathrm{id}"] & \arrow[l,swap, "k^*\otimes j^*"]\Bbbk[\G]\otimes_{\Bbbk[X]}\Bbbk[M]\arrow[d,"\mmm^*\otimes\mathrm{id}"]\\
			\Bbbk[\mathcal{H}]\otimes_{\Bbbk[S]}\Bbbk[\mu^{-1}(S)]\arrow[r,"\mathrm{id}\otimes\overline{\varphi}"] & \Bbbk[\mathcal{H}]\otimes_{\Bbbk[S]}\Bbbk[\mathcal{H}]\otimes_{\Bbbk[S]}\Bbbk[\mu^{-1}(S)] & \arrow[l,swap,"k^*\otimes k^*\otimes j^*"]\Bbbk[\mathcal{G}]\otimes_{\Bbbk[X]}\Bbbk[\G]\otimes_{\Bbbk[X]}\Bbbk[M]\\
			\Bbbk[\G]\otimes_{\Bbbk[X]}\Bbbk[M]\arrow[r,"\mathrm{id}\otimes\mathcal{A}^*"]\arrow[u, swap,"k^*\otimes j^*"] & \Bbbk[\mathcal{G}]\otimes_{\Bbbk[X]}\Bbbk[\G]\otimes_{\Bbbk[X]}\Bbbk[M]\arrow[u,swap,"k^*\otimes k^*\otimes j^*"]
		\end{tikzcd}.$$
		In the latter, the top and bottom squares in the first column and square in the second column commute. It remains to verify that the middle square in the first column commutes. To this end, suppose that $f\in\Bbbk[\mu^{-1}(S)]$. Choose $h\in\Bbbk[M]$ satisfying $f=j^*h$. We have
		\begin{align*}
			(\tau^*\otimes\mathrm{id})(\overline{\varphi}(f)) & = (\tau^*\otimes\mathrm{id})((k^*\otimes j^*)(\mathcal{A}^*h))\\
			& = (k^*\otimes k^*\otimes j^*)((\mathrm{id}\otimes\mathcal{A}^*)(\mathcal{A}^*h))\\
			& = (\mathrm{id}\otimes\overline{\varphi})((k^*\otimes j^*)(\mathcal{A}^*h))\\
			& = (\mathrm{id}\otimes\overline{\varphi})(\overline{\varphi}(j^*h))\\
			& = (\mathrm{id}\otimes\overline{\varphi})(\overline{\varphi}(f)).
		\end{align*}
		This completes the proof of (ii).
	\end{proof}
	
	\subsection{Residual actions on quotients}
	Let $\G \tto X$ and $\I \tto Y$ be affine algebraic Lie groupoids and $(\mu, \nu) : M \longrightarrow X \times Y$ an affine $\G \times \I$-scheme.
	Suppose that $\H \tto S$ is a subgroupoid of $\mathcal{I} \tto Y$ over a closed subvariety $S \s Y$. Consider the subgroupoids $\uuu_X\times\mathcal{H}\tto X\times S$ and $\G\times \uuu_S\tto X\times S$ of $\G\times\mathcal{I}$. Proposition \ref{Proposition: Action on subschemes} implies that $\nu^{-1}(S)$ is simultaneously an affine $\uuu_X\times\mathcal{H}$-scheme and an affine $\G\times \uuu_S$-scheme. These structures amount to $\nu^{-1}(S)$ being an affine $\mathcal{H}$-scheme and an affine $\G$-scheme, respectively. Write $$\varphi:\Bbbk[\nu^{-1}(S)]\longrightarrow\Bbbk[\G]\otimes_{\Bbbk[X]}\Bbbk[\nu^{-1}(S)]$$ for the algebra morphism induced by the $\G$-action on $\nu^{-1}(S)$, and $\mu':\nu^{-1}(S)\longrightarrow X$ for the restriction of $\mu$ to $\nu^{-1}(S)$. We also set $\nu^{-1}(S)\sll{}\mathcal{H}\coloneqq\mathrm{Spec}(\Bbbk[\nu^{-1}(S)]^{\mathcal{H}})$, and let $\overline{\mu}:\nu^{-1}(S)\sll{}\mathcal{H}\longrightarrow X$ denote the descent of $\mu'$ to $\nu^{-1}(S)\sll{}\mathcal{H}$.     
	
	\begin{lemma}\label{Lemma: Quotients}
		The following statements are true.
		\begin{itemize}
			\item[\textup{(i)}] One has $\varphi(\Bbbk[\nu^{-1}(S)]^{\mathcal{H}})\s\Bbbk[\G]\otimes_{\Bbbk[X]}\Bbbk[\nu^{-1}(S)]^{\mathcal{H}}$.
			\item[\textup{(ii)}] Consider the restriction of $\varphi$ to an algebra morphism $\overline{\varphi}:\Bbbk[\nu^{-1}(S)]^{\mathcal{H}}\longrightarrow \Bbbk[\G]\otimes_{\Bbbk[X]}\Bbbk[\nu^{-1}(S)]^{\mathcal{H}}$. One has $\overline{\varphi}=\overline{\mathcal{A}}^*$ for a unique affine $\mathcal{\G}$-scheme structure $\overline{\mathcal{A}}:\mathcal{G}\fp{\sss}{\overline{\mu}}(\nu^{-1}(S)\sll{}\mathcal{H})\longrightarrow\nu^{-1}(S)\sll{}\mathcal{H}$ on $\overline{\mu}:\nu^{-1}(S)\sll{}\mathcal{H}\longrightarrow X$.
		\end{itemize}
	\end{lemma}
	
	\begin{proof}
		We begin by proving (i). Equip $\G$ with the trivial $\mathcal{H}$-action, i.e. corresponding to the morphism
		$$\Bbbk[\G]\longrightarrow(\Bbbk[\G]\otimes\Bbbk[S])\otimes_{\Bbbk[S]}\Bbbk[\mathcal{H}],\quad f\mapsto (f\otimes 1)\otimes 1.$$ The $\mathcal{H}$-actions on $\G$ and $\nu^{-1}(S)$ determine a diagonal action of $\mathcal{H}$ on the fibered product $\G\fp{\sss}{\mu'}\nu^{-1}(S)$. Since the actions of $\uuu_X\times\mathcal{H}\tto X\times S$ and $\G\times \uuu_S\tto X\times S$ on $\nu^{-1}(S)$ commute with one another, $\varphi$ corresponds to an $\mathcal{H}$-equivariant morphism $\G\fp{\sss}{\mu'}\nu^{-1}(S)\longrightarrow\nu^{-1}(S)$. Lemma \ref{Lemma: Invariants} and the $\mathcal{H}$-invariance of $\nu^{-1}(S)\longrightarrow X$ imply that $$\varphi(\Bbbk[\nu^{-1}(S)]^{\mathcal{H}})\s(\Bbbk[\G]\otimes_{\Bbbk[X]}\Bbbk[\nu^{-1}(S)])^{\mathcal{H}}=\Bbbk[\G]\otimes_{\Bbbk[X]}\Bbbk[\nu^{-1}(S)]^{\mathcal{H}}.$$
		This completes the proof of (i).
		
		We now verify (ii). A first step is to consider the commutative diagram
		$$\begin{tikzcd}
			\Bbbk[X]\arrow[r,"(\mu')^*"] \arrow[d,"\ttt^*"] & \Bbbk[\nu^{-1}(S)]\arrow[d,"\varphi"] \\
			\Bbbk[\G]\arrow[r, "\mathrm{id}\otimes 1"] & \Bbbk[\G]\otimes_{\Bbbk[X]}\Bbbk[M]\\
		\end{tikzcd}$$
		and diagram
		$$\begin{tikzcd}
			& \Bbbk[\nu^{-1}(S)]\\
			\Bbbk[X]\arrow[r,"\overline{\mu}^*"]\arrow[d,"t^*"]\arrow[ur,"(\mu')^*"] & \Bbbk[\nu^{-1}(S)]^{\mathcal{H}}\arrow[d,"\overline{\varphi}"]\arrow[u]\arrow[r] & \Bbbk[\nu^{-1}(S)]\arrow[d, "\varphi"]\\
			\Bbbk[\G]\arrow[r,"\mathrm{id}\otimes 1"] & \Bbbk[\G]\otimes_{\Bbbk[X]}\Bbbk[\nu^{-1}(S)]^{\mathcal{H}}\arrow[r] & \Bbbk[\G]\otimes_{\Bbbk[X]}\Bbbk[\nu^{-1}(S)]
		\end{tikzcd},$$
		in which each unlabeled arrow is an inclusion map.
		Note that the triangle and rightmost square in the latter diagram commute. To prove that the leftmost square commutes, suppose that $f\in\Bbbk[X]$. We have 
		$$\overline{\varphi}(\overline{\mu}^*f) = \varphi((\mu')^*f) = (\mathrm{id}\otimes 1)(\ttt^*f),$$
		as desired.
		
		Let us now consider the commuting squares
		$$\begin{tikzcd}
			\Bbbk[\nu^{-1}(S)]^{\mathcal{H}}\arrow[r,"\overline{\varphi}"]\arrow[d] & \Bbbk[\G]\otimes_{\Bbbk[X]}\Bbbk[\nu^{-1}(S)]^{\mathcal{H}}\arrow[r,"\uuu^*\otimes\mathrm{id}"]\arrow[d] & \Bbbk[X]\otimes\Bbbk[\nu^{-1}(S)]^{\mathcal{H}}\arrow[r,"\overline{\mu}^*\mathrm{id}"]\arrow[d] & \Bbbk[\nu^{-1}(S)]^{\mathcal{H}}\arrow[d]\\
			\Bbbk[\nu^{-1}(S)]\arrow[r,"\varphi"] & \Bbbk[\G]\otimes_{\Bbbk[X]}\Bbbk[\nu^{-1}(S)]\arrow[r,"\uuu^*\otimes\mathrm{id}"] & \Bbbk[X]\otimes\Bbbk[\nu^{-1}(S)]\arrow[r,"(\mu')^*\mathrm{id}"] & \Bbbk[\nu^{-1}(S)]
		\end{tikzcd},$$
		in which each horizontal arrow is an inclusion. Composing the horizontal arrows in the bottom row yields the identity map on $\Bbbk[\nu^{-1}(S)]$. It follows that composing the horizontal arrows in the top row yields the identity map on $\Bbbk[\mu^{-1}(S)]^{\mathcal{H}}$.
		
		Now consider the commutative diagram 
		$$\begin{tikzcd}
			\Bbbk[\nu^{-1}(S)]\arrow[r,"\varphi"] \arrow[d,"\varphi"] & \Bbbk[\G]\otimes_{\Bbbk[X]}\Bbbk[\nu^{-1}(S)]\arrow[d,"\mmm^*\otimes\mathrm{id}"] \\
			\Bbbk[\G]\otimes_{\Bbbk[X]}\Bbbk[\nu^{-1}(S)]\arrow[r, "\mathrm{id}\otimes\varphi"] & \Bbbk[\G]\otimes_{\Bbbk[X]}\otimes\Bbbk[\G]\otimes_{\Bbbk[X]}\Bbbk[\nu^{-1}(S)]\\
		\end{tikzcd}$$
		and diagram
		$$\begin{tikzcd}
			\Bbbk[\nu^{-1}(S)]\arrow[r, "\varphi"] & \Bbbk[\G]\otimes_{\Bbbk[X]}\Bbbk[\nu^{-1}(S)]\\
			\Bbbk[\nu^{-1}(S)]^{\mathcal{H}}\arrow[r,"\overline{\varphi}"]\arrow[d,"\overline{\varphi}"]\arrow[u] & \Bbbk[\G]\otimes_{\Bbbk[X]}\Bbbk[\nu^{-1}(S)]^{\mathcal{H}}\arrow[d,"\mmm^*\otimes\mathrm{id}"]\arrow[u]\arrow[r] & \Bbbk[\G]\otimes_{\Bbbk[X]}\Bbbk[\nu^{-1}(S)]\arrow[d,"\mmm^*\otimes\mathrm{id}"]\\
			\Bbbk[\G]\otimes_{\Bbbk[X]}\Bbbk[\nu^{-1}(S)]^{\mathcal{H}}\arrow[r,"\mathrm{id}\otimes\overline{\varphi}"]\arrow[d] & \Bbbk[\G]\otimes_{\Bbbk[X]}\Bbbk[\G]\otimes_{\Bbbk[X]}\Bbbk[\nu^{-1}(S)]^{\mathcal{H}}\arrow[r]\arrow[d] & \Bbbk[\mathcal{G}]\otimes_{\Bbbk[X]}\Bbbk[\G]\otimes_{\Bbbk[X]}\Bbbk[\nu^{-1}(S)]\\
			\Bbbk[\G]\otimes_{\Bbbk[X]}\Bbbk[M]\arrow[r,"\mathrm{id}\otimes\varphi"] & \Bbbk[\mathcal{G}]\otimes_{\Bbbk[X]}\Bbbk[\G]\otimes_{\Bbbk[X]}\Bbbk[M]
		\end{tikzcd},$$
		in which each unlabeled arrow is an inclusion map. In the second diagram, the top and bottom squares in the first column and square in the second column commute. It remains to check that the middle square in the first column commutes as well. To this end, suppose $f\in\Bbbk[\nu^{-1}(S)]^{\mathcal{H}}$. We have
		\begin{align*}
			(\mmm^*\otimes\mathrm{id})(\overline{\varphi}(f)) & = (\mmm^*\otimes\mathrm{id})(\varphi(f))\\
			& = (\mathrm{id}\otimes\varphi)(\varphi(f))\\
			& = (\mathrm{id}\otimes\overline{\varphi})(\overline{\varphi}(f)).
		\end{align*}
		This completes the proof of (ii).
	\end{proof}
	
	\subsection{Graphs}
	Consider an affine algebraic Lie groupoid $\G\tto X$ and affine $\G$-scheme $(M,\mu,\mathcal{A})$.
	
	\begin{definition}\label{Definition: Graph}
		The \defn{graph} of $\A$ is the closed subscheme of $\G \times M \times M$ defined by the fibre product
		\[
		\Gamma_\A \coloneqq (\G \fp{\sss}{\mu} M) \fp{\A}{\mathrm{id}} M.
		\]
		We denote the corresponding ideal by $I_{\Gamma_\A} \s \Bbbk[\G]\otimes\Bbbk[M]\otimes\Bbbk[M]$.
	\end{definition}
	
	\begin{remark}\label{Remark: Generators}
		We make frequent use of the fact $I_{\Gamma_\A}$ is generated by
		\begin{equation}\label{f0962x0r}
			\bigg\{\sss^*f\otimes 1\otimes 1-1\otimes\mu^*f\otimes 1:f\in\Bbbk[X]\bigg\}\cup\bigg\{1\otimes 1\otimes h-\widehat{\mathcal{A}^*h}\otimes 1:h\in\Bbbk[M]\bigg\},
		\end{equation}
		where $\widehat{\mathcal{A}^*h}\in\Bbbk[\G]\otimes\Bbbk[M]$ is any lift of $\mathcal{A}^*h\in\Bbbk[\G]\otimes_{\Bbbk[X]}\Bbbk[M]$.
	\end{remark}
	
	%
	%
	
	We now consider this construction in the context of Subsection \ref{Subsection: Actions of subgroupoids on subschemes}. To this end, retain the objects and notation introduced in that subsection. Proposition \ref{Proposition: Action on subschemes} renders $\mu^{-1}(S)$ an affine $\H$-scheme. Write   $\mathcal{B}:\mathcal{H}\fp{\overline{\sss}}{\overline{\mu}}\mu^{-1}(S)\longrightarrow\mu^{-1}(S)$ for the underlying action morphism. One may consider the graphs $\Gamma_\A \s \G \times M \times M$ and $\Gamma_\B \s \H \times \mu^{-1}(S) \times \mu^{-1}(S)$, as well as their respective ideals $I_{\Gamma_\A}\s\Bbbk[\G]\otimes\Bbbk[M]\otimes\Bbbk[M]$ and $I_{\Gamma_\B}\s\Bbbk[\H]\otimes\Bbbk[\mu^{-1}(S)]\otimes\Bbbk[\mu^{-1}(S)]$. As $\Gamma_\B$ is a closed subscheme of $\G\times M\times M$, it corresponds to an ideal $J_{\Gamma_\B}\s \Bbbk[\G]\otimes\Bbbk[M]\otimes\Bbbk[M]$. Let us also write $I_{\H \times \mu^{-1}(S) \times \mu^{-1}(S)}\s\Bbbk[\G]\otimes\Bbbk[M]\otimes\Bbbk[M]$ for the ideal of $\H \times \mu^{-1}(S) \times \mu^{-1}(S)\s\G\times M\times M$. These considerations feature in the following result.
	
	\begin{lemma}\label{zwuzxpk5}
		We have $I_{\Gamma_\A} + I_{\H \times \mu^{-1}(S) \times \mu^{-1}(S)}\s J_{\Gamma_\B}$.
	\end{lemma}
	
	%
	%
	
	\begin{proof}
		Adopt the notation of Subsection \ref{Subsection: Actions of subgroupoids on subschemes}, and consider the surjective algebra morphism $$k^*\otimes j^*\otimes j^*:\Bbbk[\G]\otimes\Bbbk[M]\otimes\Bbbk[M]\longrightarrow\Bbbk[\H]\otimes\Bbbk[\mu^{-1}(S)]\otimes\Bbbk[\mu^{-1}(S)].$$
		The inclusion $I_{\H \times \mu^{-1}(S) \times \mu^{-1}(S)}\s J_{\Gamma_\B}$ follows immediately from observing that $J_{\Gamma_B}=(k^*\otimes j^*\otimes j^*)^{-1}(I_{\Gamma_\B})$ and $I_{\H \times \mu^{-1}(S) \times \mu^{-1}(S)}=\ker(k^*\otimes j^*\otimes j^*)$. 
		
		To prove that $I_{\Gamma_\A} \s J_{\Gamma_\B}$, suppose that $f\in\Bbbk[X]$. Consider the commutative diagrams $$\begin{tikzcd}
			\Bbbk[X]\arrow[r,"i^*"] \arrow[d,"\sss^*"] & \Bbbk[S]\arrow[d,"\overline{\sss}^*"] \\
			\Bbbk[\G]\arrow[r, "k^*"] & \Bbbk[\mathcal{H}]\\
		\end{tikzcd}\quad\text{and}\quad\begin{tikzcd}
			\Bbbk[X]\arrow[r,"i^*"] \arrow[d,"\mu^*"] & \Bbbk[S]\arrow[d,"\overline{\mu}^*"] \\
			\Bbbk[M]\arrow[r, "j^*"] & \Bbbk[\mu^{-1}(S)]\\
		\end{tikzcd}.$$ We have
		$$(k^*\otimes j^*\otimes j^*)(\sss^*f\otimes 1\otimes 1-1\otimes\mu^*f\otimes 1)=\overline{\sss}^*i^*f\otimes 1\otimes 1-1\otimes\overline{\mu}^*i^*f\otimes 1\in I_{\Gamma_\B}.$$ Let us also suppose that $h\in\Bbbk[M]$. Choose a lift $\widehat{\mathcal{A}^*h}\in\Bbbk[\G]\otimes\Bbbk[M]$ of $\mathcal{A}^*(h)\in\Bbbk[\G]\otimes_{\Bbbk[X]}\Bbbk[M]$. Since the squares in
		$$\begin{tikzcd}
			\Bbbk[M]\arrow[r,"\mathcal{A}^*"]\arrow[d,"j^*"] & \Bbbk[\G]\otimes_{\Bbbk[X]}\Bbbk[M]\arrow[d,"k^*\otimes j^*"] & \arrow[l]\Bbbk[\G]\otimes\Bbbk[M]\arrow[d,"k^*\otimes j^*"]\\
			\Bbbk[\mu^{-1}(S)]\arrow[r,"\mathcal{B}^*"] & \Bbbk[\mathcal{H}]\otimes_{\Bbbk[S]}\Bbbk[\mu^{-1}(S)] & \arrow[l]\Bbbk[\mathcal{H}]\otimes\Bbbk[\mu^{-1}(S)]
		\end{tikzcd}$$
		commute, $\widehat{\mathcal{B}^*j^*h}\coloneqq (k^*\otimes j^*)(\widehat{\mathcal{A}^*h})\in\Bbbk[\mathcal{H}]\otimes\Bbbk[\mu^{-1}(S)]$ is a lift of $\mathcal{B}^*j^*h\in\Bbbk[\H]\otimes_{\Bbbk[S]}\Bbbk[\mu^{-1}(S)]$. We conclude that $$(k^*\otimes j^*\otimes j^*)(1\otimes 1\otimes h-\widehat{\mathcal{A}^*h}\otimes 1)=1\otimes 1\otimes j^*h-\widehat{\mathcal{B}^*j^*h}\otimes 1\in\mathcal{I}_{\Gamma_{\mathcal{B}}}.$$ These last few sentences imply that $(k^*\otimes j^*\otimes j^*)^{-1}(I_{\Gamma_\B})=J_{\Gamma_\B}$ contains a generating set for $I_{\Gamma_\A}$, i.e. $I_{\Gamma_\A}\s J_{\Gamma_\B}$. It follows that $I_{\Gamma_\A} + I_{\H \times \mu^{-1}(S) \times \mu^{-1}(S)}\s J_{\Gamma_\B}$. 
	\end{proof}

	\subsection{Actions of products}
	Consider affine algebraic Lie groupoids $\G\tto X$ and $\I\tto Y$, as well as an affine $\G\times\I$-scheme $(\mu,\nu):M\longrightarrow X\times Y$. The action morphism $\mathcal{A}$ determines an algebra morphism $$\A^*:\Bbbk[M]\longrightarrow(\Bbbk[\G]\otimes\Bbbk[\I])\otimes_{(\Bbbk[X]\otimes\Bbbk[Y])}\Bbbk[M].$$ On the other hand, consider the identity bisection $\uuu:Y\longrightarrow\mathcal{I}$ and inclusion map $i:\uuu(Y)\longrightarrow\I$.
	
	\begin{lemma}\label{Lemma: Explicit}
		There is a unique algebra isomorphism $$\varphi:\Bbbk[\G]\otimes_{\Bbbk[X]}\Bbbk[M]\longrightarrow (\Bbbk[\G]\otimes\Bbbk[\uuu(Y)])\otimes_{(\Bbbk[X]\otimes\Bbbk[Y])}\Bbbk[M]$$ satisfying $$\varphi(f\otimes h)=f\otimes 1\otimes h$$ for all $f\in\Bbbk[\G]$ and $h\in\Bbbk[M]$. Its inverse is determined by the property that $$\varphi^{-1}(f\otimes i^*h\otimes k)=f\otimes (\uuu\circ\nu)^*(h)k$$ for all $f\in\Bbbk[\G]$, $h\in\Bbbk[\I]$, and $k\in\Bbbk[M]$.
	\end{lemma}
	
	\begin{proof}
		Consider the source morphism $\sss:\mathcal{I}\longrightarrow Y$. The composite morphism $\sss\circ i:\uuu(Y)\longrightarrow Y$ is an isomorphism. It follows that $i^*\circ\sss^*:\Bbbk[Y]\longrightarrow\Bbbk[\uuu(Y)]$ is an algebra isomorphism. On the other hand, $i^*\circ\sss^*$ renders $\Bbbk[\uuu(Y)]$ an algebra over $\Bbbk[Y]$. This is the $\Bbbk[Y]$-algebra structure on $\Bbbk[\uuu(Y)]$ with which one computes $(\Bbbk[\G]\otimes\Bbbk[\uuu(Y)])\otimes_{(\Bbbk[X]\otimes\Bbbk[Y])}\Bbbk[M]$ in the statement of the lemma. The first sentence of this lemma now follows immediately. One also has
		$$\varphi^{-1}(f\otimes i^*h\otimes k)=f\otimes ((i^*\circ\sss^*)^{-1}(i^*h)\cdot k)$$ for all $f\in\Bbbk[\G]$, $h\in\Bbbk[\I]$, and $k\in\Bbbk[M]$, where $\cdot:\Bbbk[Y]\times\Bbbk[M]\longrightarrow\Bbbk[M]$ is the $\Bbbk[Y]$-algebra structure on $\Bbbk[M]$. Note also that $\uuu=i\circ\uuu'$ for a unique isomorphism $\uuu':Y\longrightarrow\uuu(Y)$. Since $(\sss\circ i)^{-1}=\uuu'$ and $q\cdot r=\nu^*(q)r$ for all $q\in\Bbbk[Y]$ and $r\in\Bbbk[M]$, we must have 
		$$\varphi^{-1}(f\otimes i^*h\otimes k)=f\otimes (\nu^*(\uuu')^*i^*h)k=f\otimes(\nu^*\uuu^*h)k=f\otimes(\uuu\circ\nu)^*(h)k.$$ This completes the proof.
	\end{proof}
	
	We have the algebra morphism $$\eta\coloneqq\varphi^{-1}\circ(\mathrm{id}\otimes i^*\otimes\mathrm{id})\circ\mathcal{A}^*:\Bbbk[M]\longrightarrow\Bbbk[\G]\otimes_{\Bbbk[X]}\Bbbk[M].$$
	
	\begin{lemma}\label{Lemma: Factor}
		One has $\eta=\mathcal{B}^*$ for a unique affine $\G$-scheme structure $\mathcal{B}:\G\fp{\sss}{\mu}M\longrightarrow M$ on $\mu:M\longrightarrow X$.
	\end{lemma}
	
	\begin{proof}
		A first step is to consider the diagram
		$$\begin{tikzcd}
			& \Bbbk[X]\otimes\Bbbk[Y]\arrow[d,"\mu^*\nu^*"]\arrow[r,"t_{\G}^*\otimes t_{\I}^*"] & \Bbbk[\G]\otimes\Bbbk[\I]\arrow[d,"\mathrm{id}\otimes 1"]\\
			\Bbbk[X]\arrow[ur,"\mathrm{id}\otimes 1"]\arrow[r,"\mu^*"]\arrow[d,"t_{\G}^*"] & \Bbbk[M]\arrow[d,"\eta"]\arrow[r,"\mathcal{A}^*"] & (\Bbbk[\G]\otimes\Bbbk[\I])\otimes_{\Bbbk[X]\otimes\Bbbk[Y]}\Bbbk[M]\arrow[d,"\mathrm{id}\otimes i^*\otimes\mathrm{id}"]\\
			\Bbbk[\mathcal{G}]\arrow[r,"\mathrm{id}\otimes 1"] & \Bbbk[\mathcal{G}]\otimes_{\Bbbk[X]}\Bbbk[M]\arrow[r,"\varphi"] & \Bbbk[\G]\otimes\Bbbk[\uuu_Y])\otimes_{(\Bbbk[X]\otimes\Bbbk[Y])}\Bbbk[M]
		\end{tikzcd},$$
		where $t_{\G}:\G\longrightarrow X$ and $t_{\I}:\I\longrightarrow Y$ are the target morphisms of $\G$ and $\I$, respectively. Note that the triangle and two squares in the right-hand column commute. To verify that the remaining square commutes, suppose that $f\in\Bbbk[X]$. We have 
		\begin{align*}
			\varphi(\eta(\mu^*(f))) & = (\mathrm{id}\otimes i^*\otimes\mathrm{id})(\mathcal{A}^*(\mu^*(f)))\\
			& = (\mathrm{id}\otimes i^*\otimes\mathrm{id})((\mathrm{id}\otimes 1)(t_{\G}^*(f)\otimes 1))\\
			& = (\mathrm{id}\otimes i^*\otimes\mathrm{id})(t_{\G}^*(f)\otimes 1\otimes 1)\\
			& = t_{\G}^*(f)\otimes 1\otimes 1\\
			& = \varphi((\mathrm{id}\otimes 1)(t_{\G}^*(f))).
		\end{align*}
		Since $\varphi$ is injective, $\eta(\mu^*(f))=(\mathrm{id}\otimes 1)(t_{\G}^*(f))$. 
		
		Now consider the diagram
		$$\begin{tikzcd}
			\Bbbk[M]\arrow[r,"\mathcal{A}^*"] & (\Bbbk[\G]\otimes\Bbbk[\I])\otimes_{(\Bbbk[X]\otimes\Bbbk[Y])}\Bbbk[M]\arrow[d,"\mathrm{id}\otimes i^*\otimes\mathrm{id}"]\\
			& (\Bbbk[\mathcal{G}]\otimes\Bbbk[\uuu_Y])\otimes_{(\Bbbk[X]\otimes\Bbbk[Y])}\Bbbk[M]\arrow[r,"k^*\otimes\mathrm{id}"] & \Bbbk[X]\otimes\Bbbk[Y]\otimes\Bbbk[M]\arrow[r,"\mu^*\nu^*\mathrm{id}"] & \Bbbk[M]\\
			& \Bbbk[\G]\otimes_{\Bbbk[X]}\Bbbk[M]\arrow[u,"\varphi"]\arrow[r,"j^*\otimes\mathrm{id}"] & \Bbbk[X]\otimes\Bbbk[M]\arrow[u,"\mathrm{id}\otimes 1\otimes\mathrm{id}"]\arrow[ur,"\mu^*\mathrm{id}"]
		\end{tikzcd},$$
		where $j:X\longrightarrow\G$ and $k:X\times Y\longrightarrow \G\times \uuu_Y$ are the identity bisection morphisms for $\G$ and $\G\times \uuu_Y$, respectively.
		We have \begin{align*}
			(\mu^*\mathrm{id})\circ(j^*\otimes\mathrm{id})\circ\eta & = (\mu^*\nu^*\mathrm{id})\circ(\mathrm{id}\otimes 1\otimes\mathrm{id})\circ(j^*\otimes\mathrm{id})\circ\eta\\
			& = (\mu^*\nu^*\mathrm{id})\circ(k^*\otimes\mathrm{id})\circ\varphi\circ\mathrm{B}^*\\
			& = (\mu^*\nu^*\mathrm{id})\circ(k^*\otimes\mathrm{id})\circ(\mathrm{id}\otimes i^*\otimes\mathrm{id})\circ\mathcal{A}^*.
		\end{align*}
		On the other hand, Proposition \ref{Proposition: Action on subschemes} tells us that $$\mathcal{C}^*\coloneqq(\mathrm{id}\otimes i^*\otimes\mathrm{id})\circ\mathcal{A}^*$$ defines an affine $\G\times \uuu_Y$-scheme structure on $(\mu,\nu):M\longrightarrow X\times Y$. It follows that $$(\mu^*\nu^*\mathrm{id})\circ(k^*\otimes\mathrm{id})\circ\mathcal{C}^*=(\mu^*\nu^*\mathrm{id})\circ(k^*\otimes\mathrm{id})\circ(\mathrm{id}\otimes i^*\otimes\mathrm{id})\circ\mathcal{A}^*=(\mu^*\mathrm{id})\circ(j^*\otimes\mathrm{id})\circ\eta$$ is the identity on $\Bbbk[M]$.
		
		We now consider the unique algebra isomorphism $$\psi:\Bbbk[\G]\otimes_{\Bbbk[X]}\Bbbk[\G]\otimes_{\Bbbk[X]}\Bbbk[M]\longrightarrow(\Bbbk[\G]\otimes\Bbbk[\uuu_Y])\otimes_{(\Bbbk[X]\otimes\Bbbk[Y])}\otimes(\Bbbk[\G]\otimes\Bbbk[\uuu_Y])\otimes_{(\Bbbk[X]\otimes\Bbbk[Y])}\Bbbk[M]$$ satisfying $\psi(f_1\otimes f_2\otimes h)=f_1\otimes 1\otimes f_2\otimes 1\otimes h$ for all $f_1,f_2\in\Bbbk[\G]$ and $h\in\Bbbk[M]$. Let us also consider the comorphisms
		$$\theta^*:(\Bbbk[\G]\otimes\Bbbk[\uuu_Y])\otimes_{(\Bbbk[X]\otimes\Bbbk[Y])}(\Bbbk[\G]\otimes\Bbbk[\uuu_Y])\longrightarrow\Bbbk[\G]\otimes\Bbbk[\uuu_Y]\quad\text{and}\quad \tau^*:\Bbbk[\G]\otimes_{\Bbbk[X]}\Bbbk[\G]\longrightarrow\Bbbk[\G]$$ corresponding to multiplication in $\G\times \uuu_Y$ and $\G$, respectively. These maps fit into the diagram 
		$$\begin{tikzcd}
			\Bbbk[M]\arrow[r,"\mathcal{C}^*"] \arrow[d,"\mathcal{C}^*"] & (\Bbbk[\G]\otimes\Bbbk[\uuu_Y])\otimes_{(\Bbbk[X]\otimes\Bbbk[Y])}\Bbbk[M]\arrow[d,"\theta^*\otimes\mathrm{id}"] \\
			(\Bbbk[\G]\otimes\Bbbk[\uuu_Y])\otimes_{(\Bbbk[X]\otimes\Bbbk[Y])}\Bbbk[M]\arrow[r, "\mathrm{id}\otimes\mathcal{C}^*"]\arrow[d,"\varphi^{-1}"] & (\Bbbk[\G]\otimes\Bbbk[\uuu_Y])\otimes_{(\Bbbk[X]\otimes\Bbbk[Y])}\otimes(\Bbbk[\G]\otimes\Bbbk[\uuu_Y])\otimes_{(\Bbbk[X]\otimes\Bbbk[Y])}\Bbbk[M]\arrow[d,"\psi^{-1}"]\\
			\Bbbk[\G]\otimes_{\Bbbk[X]}\Bbbk[M]\arrow[r,"\mathrm{id}\otimes\eta"] & \Bbbk[\G]\otimes_{\Bbbk[X]}\Bbbk[\G]\otimes_{\Bbbk[X]}\Bbbk[M]
		\end{tikzcd}.$$
		Note that the upper square commutes. On the other hand, suppose that $f\in\Bbbk[\G]$ and $h\in\Bbbk[M]$. We have
		\begin{align*}\psi((\mathrm{id}\otimes\eta)(f\otimes h)) & = \psi(f\otimes\eta(h))\\
			& = \psi(f\otimes \varphi^{-1}(\mathcal{C}^*(h)))\\
			& = f\otimes 1\otimes\mathcal{C}^*(h)\\
			& = (\mathrm{id}\otimes\mathcal{C}^*)(f\otimes 1\otimes h)\\
			& = (\mathrm{id}\otimes\mathcal{C}^*)(\varphi(f\otimes h)).
		\end{align*}
		It follows that the lower square also commutes. We also have the commutative diagram 
		$$\begin{tikzcd}
			\Bbbk[\G]\otimes_{\Bbbk[X]}\Bbbk[M]\arrow[r,"\varphi"]\arrow[d,"\tau^*\otimes\mathrm{id}"] & (\Bbbk[\G]\otimes\Bbbk[\uuu_Y])\otimes_{(\Bbbk[X]\otimes\Bbbk[Y])}\Bbbk[M] \arrow[d,"\theta^*\otimes\mathrm{id}"]\\
			\Bbbk[\G]\otimes_{\Bbbk[X]}\Bbbk[\G]\otimes_{\Bbbk[X]}\Bbbk[M]\arrow[r,"\psi"] & (\Bbbk[\G]\otimes\Bbbk[\uuu_Y])\otimes_{(\Bbbk[X]\otimes\Bbbk[Y])}\otimes(\Bbbk[\G]\otimes\Bbbk[\uuu_Y])\otimes_{(\Bbbk[X]\otimes\Bbbk[Y])}\Bbbk[M]
		\end{tikzcd}.$$ These last few sentences imply that
		\begin{align*}
			(\mathrm{id}\otimes\eta)\circ\eta & = \psi^{-1}\circ(\mathrm{id}\otimes\mathcal{C}^*)\circ\varphi\circ\eta\\
			& = \psi^{-1}\circ(\mathrm{id}\otimes\mathcal{C}^*)\circ\mathcal{C}^*\\
			& = \psi^{-1}\circ(\theta^*\otimes\mathrm{id})\circ\mathcal{C}^*\\
			& = (\tau^*\otimes\mathrm{id})\circ\eta.
		\end{align*}
	\end{proof}

	\section{Coisotropic reduction}\label{Section: Coisotropic reduction}
	This section is devoted to the Poisson-geometric underpinnings of Main Theorem \ref{Theorem: Main}. Using the material developed in Section \ref{Section: Properties}, we give the expected algebro-geometric definitions of affine symplectic groupoids, stabilizer subgroupoids, and Hamiltonian actions on affine Poisson schemes. We conclude with a proof of Main Theorem \ref{Theorem: Main}.
	
	\subsection{Affine Poisson schemes and coisotropic subschemes}
	An \textit{affine Poisson scheme} is an affine scheme $X = \Spec A$, together with a Poisson bracket on the commutative algebra $A$. A closed subscheme $\Spec(A/I)$ of $X$ is called \textit{coisotropic} if the ideal $I$ is a Poisson subalgebra of $A$. It is with these ingredients that one can define affine symplectic groupoids and Hamiltonian actions thereof, as we explain below. 
	
	\subsection{Affine symplectic groupoids}
	Let $\G\tto X$ be an affine algebraic Lie groupoid. Consider the graph $\Gamma_{\mmm}\s\G\times\G\times\G$ of groupoid multiplication $\mmm:\G\fp{\sss}{\ttt}\G\longrightarrow\G$; it is a closed subscheme of $\G\times\G\times\G$ in a manner analogous to Definition \ref{Definition: Graph}. The following definition directly imitates its differential-geometric counterpart \cite{wei:87}. 
	
	\begin{definition}\label{y7blwyew}
		We call $\G$ an \textit{affine symplectic groupoid} if $\G$ carries a symplectic form for which $\Gamma_{\mmm}$ is coisotropic in $\G\times\G\times\G^{-}$, where $\G^{-}$ results from equipping $\G$ with the negated symplectic form.
	\end{definition}
	
	In this case, the base space $X$ is a smooth affine Poisson variety, and the source and target maps are anti-Poisson and Poisson, respectively. The cotangent bundle $T^*X\longrightarrow X$ is thereby a Lie algebroid; this is the Lie algebroid of $\G$ itself.
	
	\subsection{Stabilizer subgroupoids}\label{Subsection: Stabilizer subgroupoids}
	Let $\G\tto X$ be an affine symplectic groupoid, and write $\sigma:T^*X\longrightarrow TX$ for the result of contracting the Poisson bivector field on $X$ with cotangent vectors. Suppose that $\H\tto S$ is a smooth closed subgroupoid of $\G$ over a smooth closed subvariety $S\s X$. The Lie algebroid of $\H$ is necessarily a Lie subalgebroid of $T^*X$. With this in mind, let $N^*_SX\s T^*X\big\vert_S$ denote the conormal bundle of $S$ in $X$. One has the following special case of \cite[Section 5.1.2]{cro-may:22}. 
	
	\begin{definition}\label{Definition: Stabilizer subgroupoid}
		We call $\H\tto S$ a \textit{stabilizer subgroupoid} if $\H$ is isotropic in $\G$ and has $N^*_SX\cap\sigma^{-1}(TS)$ as its Lie algebroid.
	\end{definition}
	
	There is a relatively simple characterization of stabilizer subgroupoids over coisotropic subvarieties. To this end, retain the notation and assumptions made in the paragraph preceding Definition \ref{Definition: Stabilizer subgroupoid}. Let us also assume that $S$ is coisotropic in $X$. A straightforward exercise reveals that $\H\tto S$ is a stabilizer subgroupoid if and only if $\H$ is Lagrangian in $\G$ and has $N^*_SX$ as its Lie algebroid. This observation plays a crucial role in the proof of Lemma \ref{49t12ryi}, and thereby also in the proof of Main Theorem \ref{Theorem: Main}.
	
	\subsection{Hamiltonian actions} Let $\G\tto X$ be an affine symplectic groupoid. Suppose that $M$ is an affine Poisson scheme with an affine $\G$-scheme structure $(M,\mu,\mathcal{A})$. We imitate the Mikami--Weinstein approach \cite{mik-wei:88} to define what it means for $M$ to be an \textit{affine Hamiltonian $\G$-scheme}; this yields following scheme-theoretic version of \cite[Section 5.1.3]{cro-may:22}. 
	
	\begin{definition}\label{Definition: Hamiltonian scheme}
		We call $M$ an \textit{affine Hamiltonian $\G$-scheme} if $\Gamma_{\mathcal{A}}$ is coisotropic in $\G\times M\times M^{-}$, where $M^{-}$ indicates the negated Poisson structure on $M$.
	\end{definition}
	
	The following reformulation of Definition \ref{Definition: Hamiltonian scheme} is useful. As above, $\G\tto X$ is an affine symplectic groupoid and $M$ is an affine Poisson scheme with an affine $\G$-scheme structure $(M,\mu,\mathcal{A})$.
	
	\begin{proposition}\label{jkemvck3}
		The scheme $M$ is an affine Hamiltonian $\G$-scheme if and only if the following two conditions are satisfied:
		\begin{enumerate}[label={\textup{(\arabic*)}}]
			\item[\textup{(i)}] $\mu$ is a Poisson morphism;
			\item[\textup{(ii)}] for all $h_1, h_2 \in \Bbbk[M]$ and respective lifts $\wh{\A^*h_1}, \wh{\A^*h_2} \in \Bbbk[\G] \otimes \Bbbk[M]$ of $\mathcal{A}^*h_1,\mathcal{A}^*h_2\in\Bbbk[\G]\otimes_{\Bbbk[X]}\Bbbk[M]$, $\{\wh{\A^*h_1}, \wh{\A^*h_2}\}$ is a lift of $\A^*\{h_1, h_2\}$.
		\end{enumerate}
	\end{proposition}
	
	\begin{proof}
		To prove the forward implication, suppose that the affine $\G$-scheme structure on $M$ is Hamiltonian. We first verify that (i) holds.
		Let $f_1, f_2 \in \Bbbk[M]$.
		By definition, $\sss^*\{f_1, f_2\} \otimes 1 \otimes 1 - 1 \otimes \{\mu^*f_1, \mu^*f_2\} \otimes 1$ and $\sss^*\{f_1, f_2\} \otimes 1 \otimes 1 - 1 \otimes \mu^*\{f_1, f_2\} \otimes 1$ are in $I_{\Gamma_A}$. It follows that
		\begin{equation}\label{pvs5aj01}
			1 \otimes (\{\mu^*f_1, \mu^*f_2\} - \mu^*\{f_1, f_2\}) \otimes 1 \in I_{\Gamma_A}.
		\end{equation}
		Note that the map $$\varphi : M \too \G \times M \times M,\quad x \mtoo (\uuu(\mu(x)), x, x)$$ factors through the graph $\Gamma_\A$, i.e.\ we have a commutative diagram
		\[
		\begin{tikzcd}[row sep={2em,between origins},column sep={4em,between origins}]
			& \Gamma_\A \arrow[hook]{dr} & \\
			M \arrow{ur} \arrow[swap]{rr}{\varphi} & & \G \times M \times M.
		\end{tikzcd}
		\]
		This implies that $\varphi$ pulls \eqref{pvs5aj01} back to zero, i.e.\ $\{\mu^*f_1, \mu^*f_2\} - \mu^*\{f_1, f_2\} = 0$.
		
		We now show (ii).
		Let $h_1, h_2 \in \Bbbk[M]$.
		An argument similar to above shows that
		\begin{equation}\label{wcsfczpd}
			(\wh{\A^*\{h_1, h_2\}} - \{\wh{\A^*h_1}, \wh{\A^*h_2}\}) \otimes 1 \in I_{\Gamma_\A}
		\end{equation}
		for any lift $\wh{\A^*\{h_1, h_2\}}$ of $\A^*\{h_1, h_2\}$.
		Since the map $$\G \fp{\sss}{\mu} M \longrightarrow \G \times M \times M,\quad (g, x) \mto (g, x, \A(g, x))$$ factors through $\Gamma_\A$, it pulls \eqref{wcsfczpd} back to zero, i.e.\ $\wh{\A^*\{h_1, h_2\}} - \{\wh{\A^*h_1}, \wh{\A^*h_2}\} \in I_{\G \fp{\sss}{\mu} M}$.
		This implies that $\{\wh{\A^*h_1}, \wh{\A^*h_2}\}$ is a lift of $\A^*\{h_1, h_2\}$.
		\color{black}
		
		It remains to prove the backward implication. To this end, assume that (i) and (ii) hold. Recall the generators of $I_{\Gamma_{\A}}$ in Remark \ref{Remark: Generators}. Note that $M$ is an affine Hamiltonian $\G$-scheme if and only if $I_{\Gamma_{\A}}$ contains the Poisson bracket of any two such generators, where the Poisson brackets are taken in $\Bbbk[\G]\otimes\Bbbk[M]\otimes\Bbbk[M^{-}]$. Our approach is to verify the latter of these equivalent conditions.
		
		Suppose that $f_1,f_2\in\Bbbk[X]$. Since $\sss$ (resp. $\mu$) is anti-Poisson (resp. Poisson), we have
		\begin{align*}& \{\sss^*f_1\otimes 1\otimes 1-1\otimes\mu^*f_1\otimes 1,\sss^*f_2\otimes 1\otimes 1-1\otimes\mu^*f_2\otimes 1\}\\ & =-(\sss^*\{f_1,f_2\}\otimes 1\otimes 1-1\otimes\mu^*\{f_1,f_2\}\otimes 1)\in I_{\Gamma_{\A}}.\end{align*}
		
		Now take $h_1,h_2\in\Bbbk[M]$. Choose lifts $\wh{\A^*h_1}, \wh{\A^*h_2} \in \Bbbk[\G] \otimes \Bbbk[M]$ of $\mathcal{A}^*h_1,\mathcal{A}^*h_2\in\Bbbk[\G]\otimes_{\Bbbk[X]}\Bbbk[M]$, respectively. Noting that $\{\wh{\A^*h_1},\wh{\A^*h_2}\}$ is a lift of $\A^*\{h_1, h_2\}$, we have
		\begin{align*}
			& \{1\otimes 1\otimes h_1-\wh{\A^*h_1}\otimes 1,1\otimes 1\otimes h_2-\wh{\A^*h_2}\otimes 1\}\\
			& = -(1\otimes 1\otimes\{h_1,h_2\}-\{\wh{\A^*h_1},\wh{\A^*h_2}\}\otimes 1)\in I_{\Gamma_{\A}}.\\
		\end{align*}
		
		We now suppose that $f\in\Bbbk[X]$ and $h\in\Bbbk[M]$. Choose a lift $\wh{\A^*h}\in\Bbbk[\G]\otimes\Bbbk[M]$ of $\A^*h\in\Bbbk[\G]\otimes_{\Bbbk[X]}\Bbbk[M]$. Note that
		\begin{equation}\label{Equation: RHS}\{\sss^*f\otimes 1\otimes 1-1\otimes\mu^*f\otimes 1,1\otimes 1\otimes h-\wh{\A^*h}\otimes 1\}=\{\sss^*f\otimes 1-1\otimes\mu^*f,\wh{\A^*h}\}\otimes 1.\end{equation}
		One also knows that $\sss^*f \otimes 1 - 1 \otimes \mu^*f$ is a lift of $0 = \A^*0$. By (ii), $\{\sss^*f \otimes 1 - 1 \otimes \mu^*f, \wh{\A^*h}\}$ is a lift of $\A^*\{0, h\} = 0$. It follows immediately that the right-hand side of \eqref{Equation: RHS} is in $I_{\Gamma_{\A}}$. The backward implication is now established.
	\end{proof}
	
	\subsection{Proof of Main Theorem \ref{Theorem: Main}}
	We are nearly equipped to prove Theorem \ref{Theorem: Main}. To this end, consider an affine symplectic groupoid $\G\tto X$, smooth, closed coisotropic subvariety $S\s X$, stabilizer subgroupoid $\H\tto S$, and affine Hamiltonian $\G$-scheme $(M,\mu,\mathcal{A})$. Write $\overline{\mu}:\mu^{-1}(S)\longrightarrow S$ for the pullback of $\mu$ along the inclusion $S\longrightarrow X$, $j:\mu^{-1}(S)\longrightarrow M$ for the inclusion, and $\overline{\sss}:\H\longrightarrow S$ for the source of $\H$. By Proposition \ref{Proposition: Action on subschemes}, there is an action morphism $\mathcal{B}:\mathcal{H}\fp{\overline{\sss}}{\overline{\mu}}\mu^{-1}(S)\longrightarrow\mu^{-1}(S)$. Let $I_{\mu^{-1}(S)}\s\Bbbk[M]$, $I_{\H}\s\Bbbk[\G]$, $I_{\H\times\mu^{-1}(S)}\s\Bbbk[\G]\otimes\Bbbk[M]$, $I_{\G \fp{\sss}{\mu} M}\s\Bbbk[\G]\otimes\Bbbk[M]$, $I_{\Gamma_{\A}}\s\Bbbk[\G]\otimes\Bbbk[M]\otimes\Bbbk[M]$, and $J_{\Gamma_{\mathcal{B}}}\s\Bbbk[\G]\otimes\Bbbk[M]\otimes\Bbbk[M]$ be the ideals of the closed subschemes $\mu^{-1}(S)\s M$, $\H\s\G$, $\H\times\mu^{-1}(S)\s\G\times M$, $\G \fp{\sss}{\mu} M\s\G\times M$, $\Gamma_{\mathcal{A}}\s\G\times M\times M$, and $\Gamma_{\mathcal{B}}\s\G\times M\times M$, respectively. These considerations are relevant to the following result. 
	
	\begin{lemma}\label{49t12ryi}
		Suppose that $f \in \Bbbk[M]$.
		\begin{enumerate}[label={\textup{(\roman*)}}]
			\item\label{1clk84u5}
			We have $j^*f \in \Bbbk[\mu^{-1}(S)]^\H$ if and only if $\wh{\A^*f} - 1 \otimes f \in I_{\G \fp{\sss}{\mu} M} + I_{\H \times \mu^{-1}(S)}$ for some (hence all) lifts $\wh{\A^*f} \in \Bbbk[\G] \otimes \Bbbk[M]$ of $\A^*f \in \Bbbk[\G] \otimes_{\Bbbk[X]} \Bbbk[M]$.
			\item\label{oquntth9}
			If $1 \otimes f \otimes 1 - 1 \otimes 1 \otimes f \in J_{\Gamma_\B}$, then $j^*f \in \Bbbk[\mu^{-1}(S)]^\H$.
			\item\label{1zxac2bo}
			If $j^*f \in \Bbbk[\mu^{-1}(S)]^\H$, then
			\begin{equation}\label{2anq5pjf}
				\{f, I_{\mu^{-1}(S)}\} \s I_{\mu^{-1}(S)}.
			\end{equation}
		\end{enumerate}
	\end{lemma}
	
	\begin{proof}
		We first verify \ref{1clk84u5}.
		By definition, $j^*f \in \Bbbk[\mu^{-1}(S)]^\H$ if and only if $$\B^*j^*f - 1 \otimes j^*f = 0 \in \Bbbk[\H] \otimes_{\Bbbk[S]} \Bbbk[\mu^{-1}(S)].$$
		The latter is equivalent to $l^*(\wh{\A^*f} - 1 \otimes f) = 0$, where $l : \H \fp{\overline{\sss}}{\overline{\mu}} \too \G \times M$ is inclusion.
		Since $\ker l^* = I_{\G \fp{\sss}{\mu} M} + I_{\H \times \mu^{-1}(S)}$, this concludes the proof.
		
		We now verify \ref{oquntth9}. To this end, assume that
		\[
		1 \otimes f \otimes 1 - 1 \otimes 1 \otimes f \in J_{\Gamma_\B}.
		\]
		We also have
		\[
		\wh{\A^*f} \otimes 1 - 1 \otimes 1 \otimes f \in I_{\Gamma_\A},
		\]
		where $\wh{\A^*f}\in\Bbbk[\G]\otimes\Bbbk[M]$ is any lift of $\A^*f\in\Bbbk[\G]\otimes_{\Bbbk[X]}\Bbbk[M]$.
		Lemma \ref{zwuzxpk5} then implies that
		\[
		\wh{\A^*f} \otimes 1 - 1 \otimes f \otimes 1 \in J_{\Gamma_\B}.
		\]
		
		Consider the composition
		\[
		\H \fp{\overline{\sss}}{\overline{\mu}} \mu^{-1}(S) \too \G \fp{\sss}{\mu} M \too \G \times M \times M,
		\]
		where the first map is inclusion and the second map is the graph of $\A$.
		Since the image of this map is contained in the graph of $\B$, it pulls $\wh{\A^*f} \otimes 1 - 1 \otimes f \otimes 1$ back to zero, i.e.\
		\[
		\wh{\A^*f} - 1 \otimes f \in I_{\G \fp{\sss}{\mu} M} + I_{\H \times \mu^{-1}(S)}.
		\]
		It follows from (i) that $j^*f \in \Bbbk[\mu^{-1}(S)]^\H$.
		
		We now verify \ref{1zxac2bo}. Let $f \in \Bbbk[M]$ be such that $j^*f \in \Bbbk[\mu^{-1}(S)]^\H$. Note that $I_{\mu^{-1}(S)}$ is the ideal generated by $\mu^*h$ for all $h \in I_S$, where $I_S\s\Bbbk[X]$ is the ideal of $S$. It therefore suffices to show that $\{f, \mu^*h\} \in I_{\mu^{-1}(S)}$ for all $h\in I_S$.
		
		Fix $h\in I_S$. By Definition \ref{s469vs1d}\ref{age2txk9}, we have $\A^*\mu^*h = \ttt^*h \otimes 1$. It follows that
		\begin{equation}\label{6ly20feg}
			\ttt^*h \otimes 1 \otimes 1 - 1 \otimes 1 \otimes \mu^*h \in I_{\Gamma_\A}.
		\end{equation}
		We also know that
		\begin{equation}\label{bxh97acd}
			1 \otimes 1 \otimes f - \wh{\A^*f} \otimes 1 \in I_{\Gamma_\A},
		\end{equation}
		where $\wh{\A^*f}\in\Bbbk[\G]\otimes\Bbbk[M]$ is any lift of $\A^*f\in\Bbbk[\G]\otimes_{\Bbbk[X]}\Bbbk[M]$.
		Since $\Gamma_\A$ is coisotropic in $\G \times M \times M^-$, taking the Poisson bracket of \eqref{6ly20feg} and \eqref{bxh97acd} shows that
		\begin{equation}\label{mexu89z7}
			1 \otimes 1 \otimes \{f, \mu^*h\} - \{\wh{\A^*f}, \ttt^*h \otimes 1\} \otimes 1 \in I_{\Gamma_\A}.
		\end{equation}
		
		We now claim that
		\begin{equation}\label{d0a4pzzz}
			\{\wh{\A^*f}, \ttt^*h \otimes 1\} \in I_{\H \times \mu^{-1}(S)} + I_{\G \fp{\sss}{\mu} M}.
		\end{equation}
		A first observation is that
		\[
		\wh{\A^*f} - 1 \otimes f \in I_{\H \times \mu^{-1}(S)} + I_{\G \fp{\sss}{\mu} M},
		\]
		as follows from (i).
		Without loss of generality, we can choose the lift $\wh{\A^*f}$ such that
		\[
		\wh{\A^*f} - 1 \otimes f \in I_{\H \times \mu^{-1}(S)}.
		\]
		It follows that knowing 
		\begin{equation}\label{Equation: Sufficient}
			\{I_{\H \times \mu^{-1}(S)}, \ttt^*h \otimes 1\} \s I_{\H \times \mu^{-1}(S)} + I_{\G \fp{\sss}{\mu} M}
		\end{equation}
		would suffice to establish \eqref{d0a4pzzz}. On the other hand, note that $\H$ is coisotropic in $\G$; see Subsection \ref{Subsection: Stabilizer subgroupoids}. This implies that $\{\ttt^*h, \ell\} \in I_\H$ for all $\ell\in I_{\H}$, as $\ttt^*h \in I_\H$. Assertion \eqref{Equation: Sufficient} is therefore true, yielding \eqref{d0a4pzzz}.
		
		By Lemma \ref{zwuzxpk5} and \eqref{d0a4pzzz}, we have
		\begin{equation}\label{spvgvtt7}
			\{\wh{\A^*f}, \ttt^*h \otimes 1\} \otimes 1 \in I_{\H \times \mu^{-1}(S) \times \mu^{-1}(S)} + I_{\Gamma_\A} \s J_{\Gamma_\B}.
		\end{equation}
		It follows from \eqref{mexu89z7} and \eqref{spvgvtt7} that
		\begin{equation}\label{s3omh5k4}
			1 \otimes 1 \otimes \{f, \mu^*h\} \in J_{\Gamma_\B}.
		\end{equation}
		At the same time, consider the map
		\begin{equation}\label{6zht5r9j}
			(\uuu \circ \overline{\mu}, j, j) : \mu^{-1}(S) \longrightarrow \G \times M \times M.
		\end{equation}
		Since \eqref{6zht5r9j} factors through the graph $\Gamma_\B$, it pulls \eqref{s3omh5k4} back to zero, i.e.\ $j^*\{f, \mu^*h\} = 0$. We conclude that $\{f, \mu^*h\} \in I_{\mu^{-1}(S)}$.
	\end{proof}
	
	Main Theorem \ref{Theorem: Main} is now proved as follows.
	
	\begin{theorem}\label{Theorem: Main text}
		There exists a unique Poisson bracket on $\Bbbk[\mu^{-1}(S)]^{\H}$ satisfying $\{j^*f_1,j^*f_2\}=j^*\{f_1,f_2\}$ for all $f_1,f_2\in\Bbbk[M]$ with $j^*f_1,j^*f_2\in\Bbbk[\mu^{-1}(S)]^{\H}$. 
	\end{theorem}
	
	\begin{proof}
		Uniqueness follows immediately from the fact that $j^*$ is surjective. To establish existence is to prove that
		\begin{equation}\label{Equation: Verify}
			\{j^*f_1, j^*f_2\}
			\coloneqq
			j^*\{f_1, f_2\},\quad f_1,f_2\in (j^*)^{-1}(\Bbbk[\mu^{-1}(S)]^{\H})
		\end{equation}
		gives a well-defined Poisson bracket on $\Bbbk[\mu^{-1}(S)]^{\H}$. Well-definedness amounts to the following assertions.
		\begin{enumerate}[label={\textup{(\arabic*)}}]
			\item \label{cy29x40j}
			If $j^*\tilde{f}_1 = j^*f_1$ and $j^*\tilde{f}_2 = j^*f_2$ for some $\tilde{f}_1, \tilde{f}_2, f_1, f_2 \in \Bbbk[M]$ such that $j^*\tilde{f}_1, j^*\tilde{f}_2, j^*f_1, j^*f_2 \in \Bbbk[\mu^{-1}(S)]^\H$, then $j^*\{\tilde{f}_1, \tilde{f}_2\} = j^*\{f_1, f_2\}$.
			\item \label{3cxbbwm2}
			We have $j^*\{f_1, f_2\} \in \Bbbk[\mu^{-1}(S)]^\H$ for all $f_1, f_2 \in \Bbbk[M]$ such that $j^*f_1, j^*f_2 \in \Bbbk[\mu^{-1}(S)]^\H$.
		\end{enumerate}
		Assertion \ref{cy29x40j} is an immediate consequence of Lemma \ref{49t12ryi}\ref{1zxac2bo}. 
		
		We now prove Assertion \ref{3cxbbwm2}.
		Let $f_1, f_2 \in \Bbbk[M]$ be such that $j^*f_1, j^*f_2 \in \Bbbk[\mu^{-1}(S)]^\H$.
		As in the proof of Lemma \ref{49t12ryi}\ref{1zxac2bo}, no generality is lost in assuming the following:
		\begin{equation}\label{07xy0k37}
			\wh{\A^*f_i} = 1 \otimes f_i + a_i
		\end{equation}
		for $i=1,2$ and $a_i \in I_{\H \times \mu^{-1}(S)}$, where $\wh{\A^*f_i}\in\Bbbk[\G]\otimes\Bbbk[M]$ is a lift of $\A^*f_i\in\Bbbk[\G]\otimes_{\Bbbk[X]}\Bbbk[M]$. We also know that
		\[
		\wh{\A^*f_i} \otimes 1 - 1 \otimes 1 \otimes f_i \in I_{\Gamma_\A}
		\]
		for $i=1,2$, and that $\Gamma_\A\s\G\times M\times M^{-}$ is coisotropic. It follows that
		\[
		\{\wh{\A^*f_1}, \wh{\A^*f_2}\} \otimes 1 - 1 \otimes 1 \otimes \{f_1, f_2\} \in I_{\Gamma_\A}.
		\]
		An application of \eqref{07xy0k37} then yields
		\begin{equation}\label{glzjt7q4}
			1 \otimes \{f_1, f_2\} \otimes 1 - 1 \otimes 1 \otimes \{f_1, f_2\} \in I_{\Gamma_A} + \{1 \otimes f_1, a_2\} \otimes 1 + \{a_1, 1 \otimes f_2\} \otimes 1 + \{a_1, a_2\} \otimes 1.
		\end{equation}
		Lemma \ref{49t12ryi}\ref{oquntth9} now reduces Assertion \ref{3cxbbwm2} to verifying that the right-hand size of \eqref{glzjt7q4} is in $J_{\Gamma_\B}$.
		But Lemma \ref{49t12ryi}\ref{1zxac2bo} implies that $\{1 \otimes f_1, a_2\} \otimes 1 \in I_{\H \times \mu^{-1}(S) \times \mu^{-1}(S)}$; the same is true of $\{a_1, 1 \otimes f_2\}$. We also know that $\{a_1, a_2\} \otimes 1 \in I_{\H \times \mu^{-1}(S) \times \mu^{-1}(S)}$, as $\H \times \mu^{-1}(S)$ is coisotropic in $\G \times M$.
		Assertion \ref{3cxbbwm2} now follows from Lemma \ref{zwuzxpk5}.
		
		It remains only verify that \eqref{Equation: Verify} satisfies the axioms of a Poisson bracket on $\Bbbk[\mu^{-1}(S)]^{\H}$. This is a very straightforward consequence of the Poisson bracket on $\Bbbk[M]$ satisfying these axioms. Our proof is therefore complete.  
	\end{proof}

	\section{Residual actions}\label{Section: Residual actions}
	This section is devoted to the proof of Main Theorem \ref{Theorem: Main2}, and indications of how Main Theorems \ref{Theorem: Main} and \ref{Theorem: Main2} relate to the Moore--Tachikawa conjecture \cite{moo-tac:11}. We also include some concrete examples resulting from these theorems. 
	
	\subsection{Actions of products of affine symplectic groupoids}
	One might suspect that commuting Hamiltonian actions of two affine symplectic groupoids on an affine Poisson scheme should determine a Hamiltonian action of the product groupoid. The following result makes this idea precise.
	
	\begin{proposition}\label{Proposition: Product}
		Let $\G \tto X$ and $\I \tto Y$ be affine symplectic groupoids. Suppose that an affine Poisson scheme $M$ is simultaneously a Hamiltonian $\G$-scheme and a Hamiltonian $\I$-scheme. Assume that the actions of $\G$ and $\I$ commute, and equip $M$ with the resulting affine $\G\times\I$-scheme structure. Then $M$ is an affine Hamiltonian $\G\times\I$-scheme.
	\end{proposition}
	
	\begin{proof}
		We use the characterization of Hamiltonian actions in Proposition \ref{jkemvck3}. To this end, let $\sss:\G\longrightarrow X$ and $\overline{\sss}:\I\longrightarrow Y$ denote the source morphisms of $\G$ and $\I$, respectively. Let us also write $\mu:M\longrightarrow X$ and $\nu:M\longrightarrow Y$ for the moment maps underlying the $\G$ and $\I$-actions on $M$, respectively. In addition, let $\mathcal{A}:\G\fp{\sss}{\mu}M\longrightarrow M$, $\mathcal{B}:\I\fp{\overline{\sss}}{\nu}M\longrightarrow M$, and $\mathcal{C}:(\G\times\I)\fp{\sss\times\overline{\sss}}{(\mu,\nu)}M\longrightarrow M$ denote the action morphisms for the actions of $\G$, $\I$, and $\G\times\I$, respectively.
		
		Let us first show that $(\mu, \nu) : M \longrightarrow X \times Y$ is Poisson.
		Since $\mu$ and $\nu$ are Poisson, a straightforward computation reduces us to proving that $\{\mu^*f, \nu^*g\} = 0$ for all $f \in \Bbbk[X]$ and $g \in \Bbbk[Y]$.
		Proposition \ref{jkemvck3} tells us that $\{\wh{\A^*\mu^*f}, \wh{\A^*\nu^*g}\}$ lifts $\A^*\{\mu^*f, \nu^*g\}$.
		On the other hand, $\{\wh{\A^*\mu^*f}, \wh{\A^*\nu^*g}\} = \{\ttt^*f \otimes 1, 1 \otimes \nu^*g\} = 0$. We conclude that $\A^*\{\mu^*f, \nu^*g\} = 0$.
		By pulling $\A^*\{\mu^*f, \nu^*g\}$ back along the morphism $$M \longrightarrow \G \fp{\sss}{\mu} M,\quad x \mto (\uuu(\mu(x)), x),$$ we conclude that $\{\mu^*f, \nu^*g\} = 0$. It follows that $(\mu,\nu)$ is Poisson.
		
		Now suppose that $h_1, h_2 \in \Bbbk[M]$. Our objective is to prove that $\{\wh{\Cc^*h_1}, \wh{\Cc^*h_2}\}$ lifts $\Cc^*\{h_1, h_2\}$.
		To this end, choose lifts $\wh{\mathcal{A}^*h_1}$ and $\wh{\mathcal{A}^*h_2}$ of $\mathcal{A}^*h_1$ and $\mathcal{A}^*h_2$, respectively. Write
		\[
		\wh{\A^*h_i} = \sum_{\alpha}k_{i\alpha} \otimes m_i^\alpha
		\quad\text{for } i = 1, 2,
		\]
		where all $k_{i\alpha} \in \Bbbk[\G]$ and $m_i^\alpha \in \Bbbk[M]$ for all $\alpha$. For each index $\alpha$, let us also choose lifts $\wh{\mathcal{B}^*m_{1}^{\alpha}}$ and $\wh{\mathcal{B}^*m_{2}^{\alpha}}$ of $\mathcal{B}^*m_{1}^{\alpha}$ and $\mathcal{B}^*m_{2}^{\alpha}$, respectively.
		We have \begin{align*}\{\wh{\Cc^*h_1}, \wh{\Cc^*h_2}\} & = \sum_{\alpha,\beta}\{k_{1\alpha} \otimes \wh{\B^*m_1^\alpha}, k_{2\beta} \otimes \wh{\B^*m_2^\beta}\}\\ & = \sum_{\alpha,\beta}(\{k_{1\alpha}, k_{2\beta}\} \otimes (\wh{\B^*m_1^\alpha})(\wh{\B^*m_2^\beta}) + k_{1\alpha} k_{2\beta} \otimes \{\wh{\B^*m_1^\alpha}, \wh{\B^*m_2^\beta}\}).\end{align*}
		At the same time, Proposition \ref{jkemvck3} implies that $\{\wh{\B^*m_1^\alpha}, \wh{\B^*m_2^\beta}\}$ lifts $\B^*\{m_1^\alpha, m_2^\beta\}$ for all $\alpha,\beta$. The same proposition also tells us that $$\sum_{\alpha,\beta}\{k_{1\alpha} \otimes m_1^\alpha, k_{2\beta} \otimes m_2^\beta\}$$ lifts $\A^*\{h_1, h_2\}$. We conclude that $\{\wh{\Cc^*h_1}, \wh{\Cc^*h_2}\}$ lifts \begin{align*}& \sum_{\alpha,\beta}(\{k_{1\alpha}, k_{2\beta}\} \otimes \B^*(m_1^\alpha m_2^\beta) + k_{1\alpha} k_{2\beta} \otimes \B^*\{m_1^\alpha, m_2^\beta\})\\ & = \sum_{\alpha,\beta}(\pr_\G^* \otimes \B^*)(\{k_{1\alpha} \otimes m_1^\alpha, k_{2\beta} \otimes m_2^\beta\})\\ & = \Cc^*\{h_1, h_2\},\end{align*} where $\pr_{\G}:\G\times M\longrightarrow M$ is projection to $\G$. This completes the proof.
	\end{proof}
	
	\subsection{Proof of Main Theorem \ref{Theorem: Main2}}
	Consider two affine symplectic groupoids $\G\tto X$ and $\I\tto Y$. Suppose that $\H\tto S$ is a stabilizer subgroupoid of $\I$ over a smooth, closed, coisotropic subvariety $S\s Y$. Let us also suppose that an affine Poisson scheme $M$ carries commuting Hamiltonian $\G$ and $\I$-scheme structures. Write $\mu:M\longrightarrow X$ and $\nu:M\longrightarrow Y$ for the moment maps underlying these structures, respectively. Our assumption that the two actions commute means that we get an induced action of $\G \times \I$ on $M$ with moment map $(\mu, \nu) : M \longrightarrow X \times Y$. In particular, $\mu$ is $\I$-invariant and $\nu$ is $\G$-invariant. Let $\mu':\nu^{-1}(S)\longrightarrow X$ denote the restriction of $\mu$ to $\nu^{-1}(S)$, set $M\sll{S,\H}\I\coloneqq\mathrm{Spec}(\Bbbk[\nu^{-1}(S)]^{\H})$, and write $\overline{\mu}:M\sll{S,\H}\I\longrightarrow X$ for the descent of $\mu'$ to $M\sll{S,\H}\I$. Lemma \ref{Lemma: Quotients} implies that $\overline{\mu}:M\sll{S,\H}\I\longrightarrow X$ is an affine $\G$-scheme. Write $\overline{\A}:\G\fp{\sss}{\overline{\mu}}(M\sll{S,\H}\I)\longrightarrow M\sll{S,\H}\I$ for the action morphism. At the same time, Theorem \ref{Theorem: Main text} tells us that $M\sll{S,\H}\I$ is an affine Poisson scheme. These considerations lead to the following result. 
	
	\begin{theorem}\label{Theorem: Hamiltonian}
		The affine $\G$-scheme structure on $M\sll{S,\H}\I$ is Hamiltonian.
	\end{theorem}
	
	\begin{proof}
		We use the characterization of Hamiltonian actions in Proposition \ref{jkemvck3}.
		
		To show that $\overline{\mu}$ is Poisson, suppose that $f_1,f_2\in\Bbbk[X]$, and let $j:\nu^{-1}(S)\longrightarrow M$ denote the inclusion. Note that $j^*\mu^*f_1=\overline{\mu}^*f_1$ and $j^*\mu^*f_2=\overline{\mu}^*f_2$. It follows that $$\{\overline{\mu}^*f_1,\overline{\mu}^*f_2\}=j^*\{\mu^*f_1,\mu^*f_2\}=j^*\mu^*\{f_1,f_2\}=(\mu')^*\{f_1,f_2\}=\overline{\mu}^*\{f_1,f_2\}.$$ This shows $\overline{\mu}$ to be Poisson.
		
		Now suppose that $h_1, h_2 \in \Bbbk[\nu^{-1}(S)]^\H$. Let $\wh{\overline{\A}^*h_1},\wh{\overline{\A}^*h_2}\in\Bbbk[\G]\otimes\Bbbk[\nu^{-1}(S)]^{\H}$ be lifts of $\overline{\A}^*h_1,\overline{\A}^*h_2\in\Bbbk[\G]\otimes_{\Bbbk[X]}\Bbbk[\nu^{-1}(S)]^{\H}$, respectively.
		Choose $H_1, H_2 \in \Bbbk[M]$ such that $j^*H_1 = h_1$ and $j^*H_2 = h_2$.
		
		We claim that there exist lifts $\wh{\A^*H_i}$ of $\A^*H_i$ such that $\wh{\overline{\A}^*h_i} = (\mathrm{id} \otimes j^*) \wh{\A^*H_i}$, where $i=1,2$. A first observation is that we have the commutative diagram
		\[
		\begin{tikzcd}
			\G \fp{\sss}{\mu} M \arrow[hook]{r}{m} & \G \times M \\
			\G \fp{\sss}{\mu'} \nu^{-1}(S) \arrow[hook]{r}{n} \arrow[hook]{u}{k} \arrow{d}{\pi} & \G \times \nu^{-1}(S) \arrow[hook]{u}{l} \arrow{d}{\rho} \\
			\G \fp{\sss}{\bar{\mu}} (\nu^{-1}(S)\sll{}\H) \arrow[hook]{r}{o} & \G \times (\nu^{-1}(S) \sll{} \H)
		\end{tikzcd},
		\]
		where $\pi$ and $\rho$ are induced by $\nu^{-1}(S)\longrightarrow\nu^{-1}(S)\sll{}\H$ and all other morphisms are inclusions.
		Let $A_i$ be any lift of $\A^*H_i$.
		Since $$n^*l^*A_i = k^*m^*A_i = k^*\A^*H_i = \pi^*\overline{\A}^*h_i = \pi^* o^*\wh{\overline{\A}^*h_i} = n^*\rho^*\wh{\overline{\A}^*h_i},$$ there exists $f \in I_{\G \fp{\sss}{\mu} \nu^{-1}(S)}$ such that $l^*A_i = \rho^* \wh{\overline{A}^*h_i} + f$.
		Choose $F \in I_{\G \fp{\sss}{\mu} M}$ such that $l^*F = f$.
		Then $\wh{\A^*H_i} \coloneqq A_i - F$ is a lift of $\A^*H_i$ satisfying $$(\mathrm{id} \otimes j^*)(\wh{\A^*H_i}) = l^*(A_i - F) = \rho^* \wh{\overline{\A}^*h_i} = \wh{\overline{\A}^*h_i}.$$ This verifies our claim.
		
		Let us now show that $\{\wh{\overline{\A}^*h_1}, \wh{\overline{\A}^*h_2}\}$ lifts $\overline{\A}^*\{h_1, h_2\}$, i.e.\ $o^*\{\wh{\overline{\A}^*h_1}, \wh{\overline{\A}^*h_2}\} = \overline{\A}^*\{h_1, h_2\}$.
		Since the action of $\G$ on $M$ is Hamiltonian, $\{\wh{\A^*H_1}, \wh{\A^*H_2}\}$ lifts $\A^*\{H_1, H_2\}$, i.e.\ $m^*\{\wh{\A^*H_1}, \wh{\A^*H_2}\} = \A^*\{H_1, H_2\}$.
		It follows that \begin{align*}\pi^*o^*\{\wh{\overline{\A}^*h_1}, \wh{\overline{\A}^*h_2}\} & = n^* \rho^* \{\wh{\overline{\A}^*h_1}, \wh{\overline{\A}^*h_2}\}\\ & = n^* l^* \{\wh{\A^*H_1}, \wh{\A^*H_2}\}\\ & = k^* m^* \{\wh{\A^*H_1}, \wh{\A^*H_2}\} \\ & = k^*\A^*\{H_1, H_2\}\\ & = \pi^* \overline{\A}^*\{h_1, h_2\}.\end{align*}
		It remains only to observe that $\pi^*$ is injective.
	\end{proof}
	
	
	\subsection{Composing Hamiltonian schemes}
	
	In what follows, we let $\Delta X\s X\times X$ denote the diagonal of an affine scheme $X$. This leads to the following straightforward result.
	
	\begin{lemma}\label{Lemma: Coisotropic}
		If $X$ is an affine Poisson scheme, then $\Delta X$ is coisotropic in $X^{-}\times X$.
	\end{lemma}
	
	\begin{proof}
		Let $\pi_1:X\times X\longrightarrow X$ and $\pi_2:X\times X\longrightarrow X$ denote projections to the first and second factors, respectively. The ideal of $\Delta X$ is generated by $\{\pi_1^*f\otimes 1-1\otimes\pi_2^*f:f\in\Bbbk[X]\}\s\Bbbk[X]\otimes\Bbbk[X]$. It therefore suffices to prove that the Poisson bracket of $\pi_1^*f_1\otimes 1-1\otimes\pi_2^*f_1$ and $\pi_1^*f_2\otimes 1-1\otimes\pi_2^*f_2$ in $\Bbbk[X^-] \otimes \Bbbk[X]$ is in this ideal for all $f_1,f_2\in\Bbbk[X]$. Since $\pi_1$ and $\pi_2$ are Poisson morphisms, we have
		\begin{align*}\{\pi_1^*f_1\otimes 1-1\otimes\pi_2^*f_1,\pi_1^*f_2\otimes 1-1\otimes\pi_2^*f_2\} & =-(\{\pi_1^*f_1,\pi_1^*f_2\}\otimes 1-1\otimes\{\pi_2^*f_1,\pi_2^*f_2\})\\
			& = -(\pi_1^*\{f_1,f_2\}\otimes 1-1\otimes\pi_2^*\{f_1,f_2\}).
		\end{align*}
		This completes the proof.
	\end{proof}
	
	\begin{proposition}\label{Proposition: Useful}
		If $\G\tto X$ is an affine symplectic groupoid, then $\Delta\G\tto\Delta X$ is a stabilizer subgroupoid of $\G^{-}\times\G$.
	\end{proposition}
	
	\begin{proof}
		Lemma \ref{Lemma: Coisotropic} tells us that $\Delta X$ is coisotropic in $X^{-}\times X$. We also know that $\Delta G$ is isotropic in $\G^{-}\times\G$. It therefore suffices to prove that the Lie algebroid of $\Delta\G$ is $N^*_{\Delta X}(X\times X)$, the conormal bundle of $\Delta X$ in $X$.
		
		Consider the diagonal embedding $\phi:(\G\tto X)\longrightarrow(\G\times\G\tto X\times X)$ and induced Lie algebroid morphism $\mathrm{Lie}\hspace{2pt}\phi:\mathrm{Lie}\hspace{2pt}\G\longrightarrow\mathrm{Lie}(\G\times G)$. At the same time, let $\omega$ denote the symplectic form on $\G$. The symplectic groupoid structure on $\G^- \times \G$ determines an identification $\mathrm{Lie}(\G\times\G)=T^*(X\times X)$. This identification is the Lie algebroid isomorphism
		$$\psi:\mathrm{Lie}(\G\times\G)\longrightarrow T^*(X\times X),\quad ((x_1,x_2),(v_1,v_2))\mapsto(-\omega_{x_1})(v_1,\cdot)\oplus\omega_{x_2}(v_2,\cdot).$$ We also have $$(\psi\circ\mathrm{Lie}\hspace{2pt}\phi)(x,v)=(-\omega_{x})(v,\cdot)\oplus\omega_{x}(v,\cdot)$$ for all $(x,v)\in\mathrm{Lie}\hspace{2pt}\G$. It follows that the image of $\psi\circ\mathrm{Lie}\hspace{2pt}\phi$ is contained in $N^*_{\Delta X}(X\times X)$. Since the ranks of $\mathrm{Lie}\hspace{2pt}\G$ and $N^*_{\Delta X}(X\times X)$ coincide, $N^*_{\Delta X}(X\times X)$ is necessarily the image of $\psi\circ\mathrm{Lie}\hspace{2pt}\phi$. We conclude that $N^*_{\Delta X}(X\times X)$ is the Lie algebroid of $\phi(\G)=\Delta\G$.  
	\end{proof}
	
	Let $\G\tto X$, $\I\tto Y$, and $\mathcal{K}\tto Z$ be affine symplectic groupoids. Suppose that $M$ (resp. $N$) is an affine Poisson scheme equipped with commuting Hamiltonian actions of $\G$ and $\I^{-}$ (resp. $\I$ and $\mathcal{K}^{-}$). By Proposition \ref{Proposition: Product}, $M$ (resp. $N$) is an affine Hamiltonian $\G\times\I^{-}$-scheme (resp. $\I\times\mathcal{K}^{-}$-scheme). It follows that $M\times N$ is an affine Hamiltonian scheme for $\G\times\I^{-}\times\I\times\mathcal{K}^{-}= (\G\times\mathcal{K}^{-})\times(\I^{-}\times\I)$. Theorem \ref{Theorem: Hamiltonian}, Lemma \ref{Lemma: Coisotropic}, and Proposition \ref{Proposition: Useful} allow us to form the affine Hamiltonian $\G\times\mathcal{K}^{-}$-scheme $(M\times N)\sll{\Delta\I,\Delta Y}(\I^{-}\times\I)$. In particular, an affine Hamiltonian $\G\times\I^{-}$-scheme and affine Hamiltonian $\I\times\mathcal{K}^{-}$-scheme determine an affine Hamiltonian $\G\times\mathcal{K}^{-}$-scheme. This idea underlies our definition of the so-called \textit{affine Moore--Tachikawa category}. Further details are available in \cite{cro-may:24}.
	
	\subsection{Some examples}\label{Subsection: Some examples}
	We now discuss some concrete and interesting examples of our main results.
	
	\begin{example}[Poisson reduction at level zero]\label{Example: Poisson reduction at level zero}
		Consider an affine algebraic group $G$ over $\Bbbk=\mathbb{R}$ or $\mathbb{C}$, and let $\g$ be the Lie algebra of $G$. Suppose that $\G$ is the cotangent groupoid $T^*G\tto\g^*$, $S=\{0\}\s\g^*$, and $\H$ is the isotropy group of $0$ in $\G$. Let $M$ be an affine Hamiltonian $\G$-scheme with moment map $\mu: M\longrightarrow\g^*$. This is equivalent to $M$ being an affine Poisson scheme with a Hamiltonian action of $G$ and moment map $\mu:M\longrightarrow\g^*$. We also note that $I\coloneqq\langle\mu^{\xi}:\xi\in\g\rangle\s \Bbbk[M]$ is the ideal of the closed subscheme $\mu^{-1}(0)\s M$, where $\mu^{\xi}\in\Bbbk[M]$ is the pointwise pairing of $\mu$ with $\xi\in\g$; hence $\Bbbk[\mu^{-1}(0)]=\Bbbk[M]/I$. The isotropy group $\H$ can be identified with $G$ in such a way that $\Bbbk[\mu^{-1}(0)]^{\H}=(\Bbbk[M]/I)^G$. By applying Theorem \ref{Theorem: Main text} with $\G$, $S$, and $\H$ as above, we obtain a unique Poisson bracket on $(\Bbbk[M]/I)^G$ with the following property: $\{\pi(f_1),\pi(f_2)\}=\pi(\{f_1,f_2\})$ for all $f_1,f_2\in\Bbbk[M]$ with $\pi(f_1),\pi(f_2)\in (\Bbbk[M]/I)^G$, where $\pi:\Bbbk[M]\longrightarrow\Bbbk[M]/I$ is the quotient map. For $\Bbbk=\mathbb{C}$, one should regard $\mathrm{Spec}((\Bbbk[M]/I)^G)$ as a scheme-theoretic version of Poisson reduction at level zero. If another affine algebraic group $H$ acts on $M$ in a Hamiltonian fashion, and in a manner that commutes with the $G$-action on $M$, then Theorem \ref{Theorem: Hamiltonian} equips $\mathrm{Spec}((\Bbbk[M]/I)^G)$ with a Hamiltonian $H$-action. We discuss specific instances of this construction in the next example. 
	\end{example}
	
	\begin{example}[Affine closures of cotangent bundles of homogeneous varieties]\label{Example: Affine closures}
		Let $G$ be a connected semisimple affine algebraic group over $\Bbbk=\mathbb{C}$. Note that $G\times G$ acts on $G$ by $(g_1,g_2)\cdot h=g_1hg_2^{-1}$ for $(g_1,g_2)\in G\times G$ and $h\in G$. There is an induced Hamiltonian action of $G\times G$ on $T^*G$. We may restrict this to a Hamiltonian action of $H=\{e\}\times H \subset G\times G$. The Marsden--Weinstein reduction of $T^*G$ by $H$ at level $0$ is canonically symplectomorphic to $T^*(G/H)$. On the other hand, let $I\subset\mathbb{C}[T^*G]$ be the ideal generated by the components of the $H$-moment map on $T^*G$. There is a canonical $H$-equivariant algebra isomorphism $\mathbb{C}[T^*G]/I\cong\mathbb{C}[G\times\h^{\circ}]$, where $\h\subset\g$ is the Lie algebra of $H$, $\h^{\circ}\subset\g^*$ is the annihilator of $\h$ in $\g^*$, and the $H$-action on $\mathbb{C}[G\times\h^{\circ}]$ is induced by the following $H$-action on $G\times\h^{\circ}$: $h\cdot (g,\xi)=(gh^{-1},\mathrm{Ad}_h^*(\xi))$ for $h\in H$ and $(g,\xi)\in G\times\h^{\circ}$. We may apply Example \ref{Example: Poisson reduction at level zero} to obtain the affine Poisson scheme $$\overline{T^*(G/H)}\coloneqq\mathrm{Spec}(\mathbb{C}[G\times\h^{\circ}]^H).$$ By means of Theorem \ref{Theorem: Hamiltonian}, the action of $G=G\times\{e\}$ on $T^*G$ renders $\overline{T^*(G/H)}$ an affine Hamiltonian $G$-scheme. One typically calls $\overline{T^*(G/H)}$ the \textit{affinization} or \textit{affine closure} of $T^*(G/H)$. A particularly noteworthy example arises when $H=U$ is a maximal unipotent subgroup of $G$. The affine closure $\overline{T^*(G/U)}$ occurs naturally in the theories of \textit{hyperk\"ahler implosion} \cite{dan-kir-swa} and \textit{symplectic singularities} \cite{gan,gin-kaz:22}. 
	\end{example}
	
	\begin{example}[Moore--Tachikawa schemes]
		Retain the notation of the previous example. Use the Killing form to freely identify $\g^*$ with $\g$, and the left trivialization to freely identify $T^*G$ with $G\times\g^*=G\times\g$. Set $\ell\coloneqq\mathrm{rank}\hspace{2pt}\g$ and $\greg\coloneqq\{x\in\g:\dim\g_x=\ell\}$. Let $(e,h,f)\in\g^{\times 3}$ be an $\mathfrak{sl}_2$-triple with $e,h,f\in\greg$. One may consider the associated \textit{Kostant slice} $\mathcal{S}\coloneqq e+\g_f\subset\g$. Note that $\mathcal{N}\coloneqq G\times\mathcal{S}\subset G\times\g=T^*G$ is a symplectic subvariety with the property of being invariant under the action of $G=G\times\{e\}\subset G\times G$ on $G\times\g=T^*G$. We also have the \textit{universal centralizer}
		$$\mathcal{J}\coloneqq\{(g,x)\in\mathcal{N}:g\in G_x\}.$$
		
		Given $n\in\mathbb{Z}_{\geq 1}$, consider the Hamiltonian action of $G^n\times G^n$ on $T^*(G^n)=(G\times\g)^n$. Write $\mu:(G\times\g)^n\longrightarrow\g^n$ for the moment map associated with the Hamiltonian action of $G^n=\{e\}\times G^n\subset G^n\times G^n$. At the same time, let $\Delta_n\mathcal{S}\subset\g^n$ denote the diagonally embedded copy of $\mathcal{S}$. Let $\mathcal{J}_n$ denote the kernel of multiplication $$\underbrace{\mathcal{J}\times_{\mathcal{S}}\cdots\times_{\mathcal{S}}\mathcal{J}}_{n\text{ times}}\longrightarrow\mathcal{J},$$
		and set
		$$\mathcal{N}_n\coloneqq \underbrace{\mathcal{N}\times_{\mathcal{S}}\cdots\times_{\mathcal{S}}\mathcal{N}}_{n\text{ times}}.$$ 
		Note that $\mu^{-1}(\Delta_n\mathcal{S})=\mathcal{N}_n$, and that $\mathcal{J}_n$ acts on $\mathcal{N}_n$; see \cite[Subsection 5.4]{cro-may:22} for more details. Theorem \ref{Theorem: Main text} is readily seen to imply that $\mathrm{Spec}(\mathbb{C}[\mathcal{N}_n]^{\mathcal{J}_n})$ is an affine Poisson scheme. By Theorem \ref{Theorem: Hamiltonian}, the Hamiltonian action of $G^n=G^n\times\{e\}\subset G^n\times G^n$ on $T^*(G^n)$ renders $\mathrm{Spec}(\mathbb{C}[\mathcal{N}_n]^{\mathcal{J}_n})$ an affine Hamiltonian $G^n$-scheme. It turns out that $\mathrm{Spec}(\mathbb{C}[\mathcal{N}_n]^{\mathcal{J}_n})$ is a Moore--Tachikawa scheme; see \cite{cro-may:22} for more details.
	\end{example}

										\bibliographystyle{acm}
										\bibliography{scheme-theoretic-coisotropic-reduction}
										
									\end{document}